\documentclass[12pt]{amsart}
\usepackage{amsmath,amssymb,amsthm,amscd}
\usepackage{mathrsfs,MnSymbol}
\usepackage{fullpage}
\usepackage{enumitem}
\usepackage{float}
\usepackage{caption,subcaption}   
\usepackage{tikz-cd}
\usepackage[msc-links]{amsrefs}
\usepackage{hyperref}
\hypersetup{
  colorlinks   = true,	
  urlcolor     = blue,
  linkcolor    = blue,
  citecolor    = red
}
\theoremstyle{plain}
\newtheorem{theorem}{Theorem}
\newtheorem{lemma}[theorem]{Lemma}
\newtheorem{corollary}[theorem]{Corollary}
\newtheorem{proposition}[theorem]{Proposition}
\theoremstyle{definition}

\theoremstyle{remark}
\newtheorem{remark}[theorem]{Remark}
\numberwithin{theorem}{section} \numberwithin{equation}{section}
\allowdisplaybreaks
\AtBeginDocument{%
   \def\MR#1{}
}
\begin{document}
\title{On the Vinberg Family of K3 Surfaces}
\author{Adrian Clingher}
\address{Department of Mathematics and Statistics, University of Missouri - St. Louis, St. Louis, MO 63121}
\email{clinghera@umsl.edu}
\author{Andreas Malmendier}
\address{Department of Mathematics \& Statistics, Utah State University, Logan, UT 84322}
\email{andreas.malmendier@usu.edu}
\author{Brandon Williams}
\address{Institute for Mathematics, Heidelberg University, 69120 Heidelberg, Germany}
\email{bwilliams@mathi.uni-heidelberg.de}
\begin{abstract}
We study orthogonal modular forms associated with moduli spaces of lattice-polarized K3 surfaces whose generic transcendental lattices are of the form $T = H \oplus H \oplus L(-1)$ where $L$ is a root lattice of type $A_n$ or $D_n$. In Picard numbers $10$ through $17$, we use explicit Jacobian elliptic fibrations to construct modular forms on type IV domains associated with orthogonal groups $\mathrm{O}^+(T)$. We show that the coefficients of suitable Weierstrass models naturally realize generators for the corresponding graded rings of orthogonal modular forms. 
%In particular, for $T = H \oplus H \oplus D_8(-1)$ and $T = H \oplus H \oplus A_7(-1)$, we prove that the coefficients of a canonical elliptic fibration generate the corresponding graded rings of modular forms.
\end{abstract}
\maketitle
\section{Introduction}
Let $\mathcal{X}$ be a smooth complex projective K3 surface. The intersection form gives the second cohomology group $H^2(\mathcal{X}, \mathbb{Z})$ the structure of an even integral lattice. This lattice is the unique (up to isometry) unimodular even lattice of signature $(3, 19)$, which we denote $\Lambda_{K3} \cong H^{\oplus 3} \oplus E_8(-1)^{\oplus 2}$. Denote by $\mathrm{NS}(\mathcal{X})$ the N\'eron--Severi lattice. The rank of $\mathrm{NS}(\mathcal{X})$ is called the Picard number and ranges between $1$ and $20$. 
\par In this article, we shall study some specific lattice polarizations on $\mathcal{X}$. Let $S$ be an even lattice of signature $(1,\rho_S-1)$ admitting a primitive embedding $S\hookrightarrow \Lambda_{K3}$. Following the definition in \cites{MR4635303,MR1420220,Vinberg2018}, an $S$-polarization on $\mathcal{X}$ is a primitive lattice embedding $i:S\hookrightarrow \operatorname{NS}{(\mathcal{X})}$ whose image contains a big and nef class\footnote{As explained in \cite{MR4635303} this class has to be the image of a very irrational vector of positive norm in $S\otimes \mathbb{R}$.}. Isomorphism classes of $S$-polarized K3 surfaces form a quasi-projective moduli space $\mathscr M_S$ of dimension $20-\rho_S$; see \cite{MR1420220}.
\par Let $T$ be the orthogonal complement of $S$ in $\Lambda_{K3}$. Hodge structures (periods) for $S$-polarized K3 surfaces are known (see \cite{Vinberg2018}, as well as \cites{MR2030225,MR1420220}) to be classified, up to isomorphism, by the quotient space $ \mathrm{O}^+(T) \backslash \mathscr{D}^+(T)$, where 
\begin{equation}
\label{eqn:ring_automorphic_forms}
%  \mathscr{D}^+(L) \ = \ \Big\{ [z] \in \mathbb{P}\big( L \otimes \mathbb{C} \big) \colon \ \langle z , z \rangle =0 , \ \langle z , \overline{z} \rangle > 0 \Big\}^+ 
 \mathscr{D}^+(T) \ = \ \Big\{ [Z] \in \mathbb{P}\big(T \otimes \mathbb{C} \big) \colon \ Z \cdot Z =0 , \ Z \cdot  \overline{Z}  > 0 \Big\} 
 \end{equation} 
 is the Hermitian symmetric domain of type IV associated with $T$. Here, the index $+$ indicates a choice of connected component and $\mathrm{O}^+(T)  < \mathrm{O}(T)$ denotes the spinor kernel subgroup of integral isometries that preserve the connected component $\mathscr{D}^+(T)$. 
 %then $\widetilde{\mathrm{O}}(L)$ is the discriminant kernel subgroup of $\mathrm{O}^+(L)$, i.e. the subgroup corresponding to isometries that act trivially on the discriminant group of $L$.  
 %
 \par An appropriate version of the Torelli Theorem (see \cites{Vinberg2018, Vinberg2010}, as well as \cite{MR2030225}) allows one to identify the quasi-projective moduli space $\mathscr{M}_S$ 
 with the quotient $ \mathrm{O}^+(T) \backslash \mathscr{D}^+(T)$. It should also be noted that 
 $\mathrm{O}^+(T; \mathbb{R})  = \mathrm{SO}(2,20-\rho_S)$ acts transitively on $\mathscr{D}^+(T)$, which leads to an identification between $\mathscr{D}^+(T)$ and the bounded symmetric domain  
 \begin{equation*}
 \mathrm{SO}(2,20-\rho_S )  / \mathrm{SO}(2)  \times \mathrm{O}(20-\rho_S  )  \;.
\end{equation*}
The classifying period space $ \mathrm{O}^+(T) \backslash \mathscr{D}^+(T)$ may then be seen, from this point of view, as a quotient of a bounded symmetric domain by the action of a (modular) discrete group.
\par The present paper explores the connection between the algebraic descriptions for $S$-polarized K3 surfaces and the orthogonal modular forms associated with the domain  $ \mathscr{D}^+(T)$ and group $\mathrm{O}^+(T)$ for some particularly interesting families of K3 surfaces in Picard number $\rho_S=10, \dots, 16$.

As proved by Nikulin \cites{MR544937,MR633160,MR633160b,MR752938,MR3165023}, Picard number $14$ is the highest rank where there exists more than one 2-elementary, primitive sub-lattice of $\Lambda_{K3}$ for K3 surfaces with finite automorphism groups. The three possibilities, in this rank, are
\begin{equation}
\label{eqn:lattices_intro}
 H \oplus E_8(-1) \oplus A_1(-1)^{\oplus 4} , \quad  H \oplus D_8(-1) \oplus D_4(-1), \quad  H \oplus E_8(-1) \oplus D_4(-1).
\end{equation} 
Here, $E_n(-1)$, $D_n(-1)$, $A_n(-1)$ are the negative definite even lattices associated with  their namesake root systems. 

\par In previous work, we have investigated the K3 surfaces whose polarizing lattices are given by the former two lattices (and their generalizations); see \cite{MR4987329}. In this article, we will focus on the third lattice $S_{14}$ and its generalizations to Picard  numbers 10 through 16. Special cases of these K3 surfaces have appeared in the work of Vinberg in \cites{Vinberg2010, Vinberg2018}. The dual graph of smooth rational curves for the very general K3 surface $\mathcal{X}$ polarized by $S_{14}$  was determined by Vinberg in \cite{Vinberg2010}*{Table~2} and is given by Figure~\ref{fig:pic14vin}.
\begin{figure}
\centering
% Top figure
\begin{subfigure}{\textwidth}
    \centering
    \scalebox{0.5}{%
    \begin{tikzpicture}
    \tikzstyle{every node}=[draw,shape=circle];
    \node (w0) at (0,0) 		[very thick] {$E_1$};
    \node (w1) at (2,0)  		[very thick] {$E_2$};
    \node (w2) at (4,0)  		[very thick] {$E_3$};
    \node (w3) at (6,0)  		[very thick] {$E_4$};
    \node (w4) at (8,0)   		[very thick] {$E_5$};
    \node (w5) at (10,0) 		[very thick] {$E_6$};
    \node (w6) at (12,0) 		[very thick] {$E_7$};
    \node (w7) at (-2,0) 		[very thick] {$E_8$};
    \node (w8) at (-4,0) 		[very thick] {$E_9$};
    \node (w9) at (-6,0) 		[very thick] {$E_{10}$};
    \node (w10) at (-8,0) 		[very thick] {$E_{11}$};
    \node (w11) at (-10,0) 	[very thick] {$E_{12}$};
    \node (w13) at (-6,-2)  	[very thick] {$E_{13}$};
    \node (w14) at (10,-2) 	[very thick] {$E_{14}$};
    \node (w15) at (10,2)   	[very thick] {$E_{15}$};
    \draw[thick]
    (w0) -- (w1)
    (w1) -- (w2)
    (w2) -- (w3)
    (w3) -- (w4)
    (w4) -- (w5)
    (w5) -- (w6)
    (w0) -- (w7)
    (w7) -- (w8)
    (w8) -- (w9)
    (w9) -- (w10)
    (w10) -- (w11)
    (w13) -- (w9)
    (w14) -- (w5) 
    (w15) -- (w5);
    \end{tikzpicture}}
    \caption{Rational curves on $\mathcal{X}$ with $\mathrm{NS}(\mathcal{X})=S_{14}$.}
    \label{fig:pic14vin}
\end{subfigure} \\[1em]
% Bottom left
\begin{subfigure}{0.49\textwidth}
    \centering
    \scalebox{0.3}{%
    \begin{tikzpicture}
    \tikzstyle{every node}=[draw,shape=circle];
    \node (w0) at (0,0) 		[fill=green] {$E_1$};
    \node (w1) at (2,0)  		[fill=green] {$E_2$};
    \node (w2) at (4,0)  		[fill=green] {$E_3$};
    \node (w3) at (6,0)  		[fill=red] {$E_4$};
    \node (w4) at (8,0)   		[fill=blue] {$E_5$};
    \node (w5) at (10,0) 		[fill=blue] {$E_6$};
    \node (w6) at (12,0) 		[fill=blue] {$E_7$};
    \node (w7) at (-2,0) 		[fill=green] {$E_8$};
    \node (w8) at (-4,0) 		[fill=green] {$E_9$};
    \node (w9) at (-6,0) 		[fill=green] {$E_{10}$};
    \node (w10) at (-8,0) 		[fill=green] {$E_{11}$};
    \node (w11) at (-10,0) 	[fill=green] {$E_{12}$};
    \node (w13) at (-6,-2)  	[fill=green]{$E_{13}$};
    \node (w14) at (10,-2) 	[fill=blue] {$E_{14}$};
    \node (w15) at (10,2)   	[fill=blue] {$E_{15}$};
    \draw[thick]
    (w0) -- (w1)
    (w1) -- (w2)
    (w2) -- (w3)
    (w3) -- (w4)
    (w4) -- (w5)
    (w5) -- (w6)
    (w0) -- (w7)
    (w7) -- (w8)
    (w8) -- (w9)
    (w9) -- (w10)
    (w10) -- (w11)
    (w13) -- (w9)
    (w14) -- (w5) 
    (w15) -- (w5);    
    \end{tikzpicture}}
    \caption{First fibration on $\mathcal{X}$.}
    \label{fig:pic14vin_fib1}
\end{subfigure}
\hfill
% Bottom right
\begin{subfigure}{0.49\textwidth}
    \centering
    \scalebox{0.3}{%
    \begin{tikzpicture}
    \tikzstyle{every node}=[draw,shape=circle];
    \node (w0) at (0,0) 		[fill=green] {$E_1$};
    \node (w1) at (2,0)  		[fill=green] {$E_2$};
    \node (w2) at (4,0)  		[fill=green] {$E_3$};
    \node (w3) at (6,0)  		[fill=green] {$E_4$};
    \node (w4) at (8,0)   		[fill=green] {$E_5$};
    \node (w5) at (10,0) 		[fill=green] {$E_6$};
    \node (w6) at (12,0) 		{$E_7$};
    \node (w7) at (-2,0) 		[fill=green] {$E_8$};
    \node (w8) at (-4,0) 		[fill=green] {$E_9$};
    \node (w9) at (-6,0) 		[fill=green] {$E_{10}$};
    \node (w10) at (-8,0) 		[fill=green] {$E_{11}$};
    \node (w11) at (-10,0) 	[fill=red] {$E_{12}$};
    \node (w13) at (-6,-2)  	[fill=green]{$E_{13}$};
    \node (w14) at (10,-2) 	[fill=green] {$E_{14}$};
    \node (w15) at (10,2)   	[fill=green] {$E_{15}$};
    \draw[thick]
    (w0) -- (w1)
    (w1) -- (w2)
    (w2) -- (w3)
    (w3) -- (w4)
    (w4) -- (w5)
    (w5) -- (w6)
    (w0) -- (w7)
    (w7) -- (w8)
    (w8) -- (w9)
    (w9) -- (w10)
    (w10) -- (w11)
    (w13) -- (w9)
    (w14) -- (w5) 
    (w15) -- (w5);    
    \end{tikzpicture}}
    \caption{Second fibration on $\mathcal{X}$.}
    \label{fig:pic14vin_fib2}
\end{subfigure}

\caption{Configurations of rational curves and associated fibrations on $\mathcal{X}$.}
\label{fig:combined_pic14vin}
\end{figure}
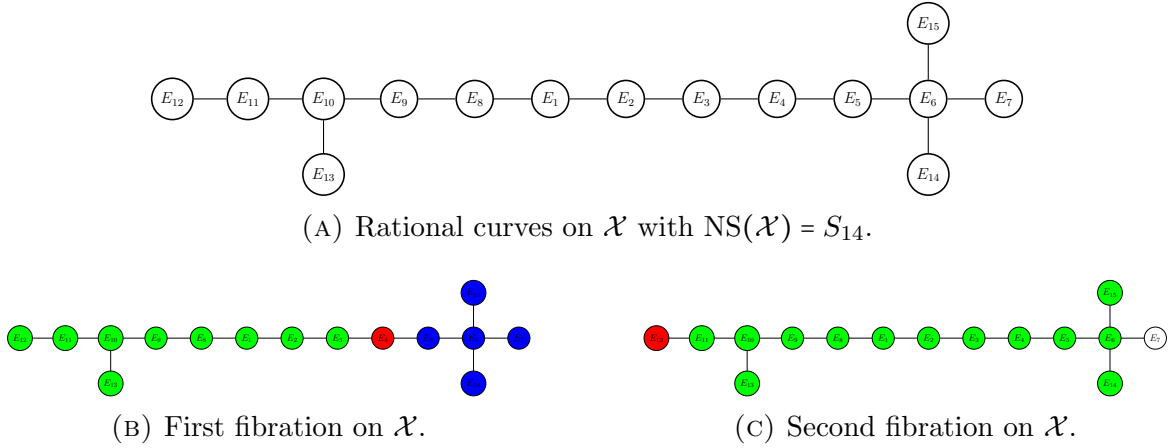
\par It is easy to construct embeddings into the graph given by Figure~\ref{fig:pic14vin} of the reducible fibers for each Jacobian elliptic fibration that is supported on such a K3 surface $\mathcal{X}$. There are two possibilities: for the first fibration, the reducible fibers are shown embedded into the dual graph in Figure~\ref{fig:pic14vin_fib1}, where the green nodes represent a reducible fiber of type ${\color{green}\widetilde{E}_8}$, the blue nodes represent a reducible fiber of type ${\color{blue}\widetilde{D}_{4}}$, and the red node represents the class of the section. Similarly, for the second fibration the reducible fibers are shown embedded into the dual graph  in Figure~\ref{fig:pic14vin_fib2}, where the green nodes represent a reducible fiber of type ${\color{green}\widetilde{D}_{12}}$ and the red node represents the class of the section. A suitable generalization of the latter fibration  is present for all K3 surfaces with transcendental lattice $T=H \oplus H \oplus D_n(-1)$. In this article, we consider the lattices $T = H \oplus H \oplus D_n(-1)$ with  $2\le n \le 8$, where $D_3 = A_3$ and $D_2 = A_1^{\oplus 2}$. As we will show, in these cases the aforementioned elliptic fibration is key in constructing the orthogonal modular forms associated with the domain  $ \mathscr{D}^+(T)$ and group $\mathrm{O}^+(T)$. Our approach will be to first construct the orthogonal modular forms associated with the period domains for a second family of lattice-polarized K3 surfaces, namely those for which the generic transcendental lattice is $H \oplus H \oplus A_n(-1)$ for $n \le 7$ in Picard numbers 11 through 18. We then use this to construct the modular forms for the Vinberg family with $T = H \oplus H \oplus D_n(-1)$ with  $2\le n \le 8$ in Picard numbers 10 through 16. 
\section*{Statement of results}
In this article we prove the following:

\begin{theorem}\label{thm:main}
\begin{enumerate}[
    label=(\roman*),
    wide,
    labelindent=0pt,
    itemsep=4pt,
    topsep=0pt,
    parsep=0pt,
    partopsep=0pt
]
\item \label{thm:partI}
The general $S$-polarized K3 surface with $T = S^{\perp} = H \oplus H \oplus D_8(-1)$ carries a unique (up to automorphism) Jacobian elliptic fibration with Weierstrass form $$y^2 = x^3 + (u^3 + F_4 u + F_6) x^2 + (G_8 u^2 + G_{10} u + G_{12})x + H_8 u^5 + H_{10} u^4 + H_{12} u^3 + H_{14} u^2 + H_{16} u+ H_{18}.$$ The coefficients $F_k, G_k, H_k$ are modular forms with respect to the orthogonal group $\mathrm{O}^+(T)$ and generate the graded ring $M(\mathrm{O}^+(T)) = M(\Gamma_{D_8})$ as in Theorem \ref{thrm:freealg}.
\item \label{thm:partII}
The general $S$-polarized K3 surface with $T = S^{\perp} = H \oplus H \oplus A_7(-1)$  carries a unique (up to automorphism) Jacobian elliptic fibration with Weierstrass form $$y^2 = x^3 + (u^3 + f_4 u + f_6) x^2 + (g_4 u^4 + g_6 u^3 + g_8 u^2 + g_{10} u + g_{12}) x + (h_5 u^2 + h_7 u + h_9)^2.$$ The coefficients $f_k, g_k, h_k$ are modular forms with respect to the orthogonal group $\mathrm{\tilde O}(T)$ and generate the graded ring $M(\mathrm{\tilde O}(T)) = M(\Gamma_{A_7}).$
\end{enumerate}
\end{theorem}
The significance of Theorem~\ref{thm:main} is that it provides a direct geometric realization of the generators of free algebras of orthogonal modular forms, in terms of explicit families of K3 surfaces of Picard number ten and eleven. While the freeness of the corresponding modular form rings was previously established by Vinberg \cites{Vinberg2010,Vinberg2018} and, more generally, by Wang--Williams \cite{WW20}, the generators themselves were not known to arise naturally from algebraic models of the associated K3 surfaces. Theorem~\ref{thm:main} shows that the coefficients of canonical Jacobian elliptic fibrations give precisely such generators for $D_8$ and $A_7$.
\section{The Vinberg family}
\subsection{Facts about Jacobian elliptic K3 surfaces}
We begin by recalling standard notions from lattice theory that will be used throughout. By a lattice, we always mean an even integral lattice, i.e. a free $\mathbb{Z}$-module $L$ of finite rank equipped with a nondegenerate integer-valued quadratic form $Q_L$. For lattices $L_1$ and $L_2$, we denote by $L_1 \oplus L_2$ their orthogonal direct sum. For $\lambda \in \mathbb{Z}$, the lattice $L(\lambda)$ is obtained from $L$ by scaling its quadratic form by $\lambda$. Given a symmetric matrix $m$, we write $\langle m \rangle$ for the lattice having $m$ as Gram matrix in a suitable basis. The symbols $A_n$, $D_n$, and $E_n$ refer to the positive-definite root lattices associated with the corresponding root systems, while $H$ denotes the unique even unimodular lattice of signature $(1,1)$.

Let $L$ be a lattice of rank $\rho$. Its discriminant group is defined as $D(L) = L^\vee / L$, equipped with the discriminant quadratic form $Q_L \colon D(L) \to \mathbb{Q}/\mathbb{Z}$. The lattice $L$ is called \emph{2-elementary} if $D(L)$ is an elementary abelian $2$-group, i.e.  $D(L) \cong (\mathbb{Z}/2\mathbb{Z})^\ell$, where $\ell$ is the length of $L$, defined as the minimal number of generators of $D(L)$. In this case one also defines the parity $\delta \in \{0,1\}$ by setting $\delta = 0$ if $Q_L$ takes values in $\frac{1}{2}\mathbb{Z}/\mathbb{Z}$ for all elements of $D(L)$, and $\delta = 1$ otherwise. A theorem of Nikulin asserts that an even, hyperbolic, 2-elementary lattice admitting a primitive embedding into the K3 lattice $\Lambda_{K3}$ is uniquely determined by $(\rho,\ell,\delta)$.

\par Let $\mathcal{X}$ be a K3 surface  and denote by $S = \mathrm{NS}(\mathcal{X})$ its N\'eron-Severi lattice or rank $\rho_S$. The transcendental lattice $T_{\mathcal{X}}$ is the orthogonal complement of $S$ within $H^2(\mathcal{X},\mathbb{Z})$. For a Jacobian elliptic surface $\mathcal{X}$, we denote the elliptic fibration by $\pi \colon \mathcal{X} \to \mathbb{P}^1$, the zero section by $\sigma$, the smooth fiber class by $f$, and the Mordell--Weil group by $\operatorname{MW}(\mathcal{X},\pi)$.  Kodaira classified all possible singular fibers in a Weierstrass model: these consist of two infinite families $(I_n, I_n^*,\, n \ge 0)$ and six exceptional types $(II, III, IV, II^*, III^*, IV^*)$; see \cite{MR1078016}. After resolution, these fibers correspond to root lattices as follows:
\[
I_n \leftrightarrow A_{n-1}, \quad
I_m^* \leftrightarrow D_{m+4}, \quad
IV^* \leftrightarrow E_6, \quad
III^* \leftrightarrow E_7, \quad
II^* \leftrightarrow E_8,
\]
while the types $III$ and $IV$ give rise to $A_1$ and $A_2$, respectively.  A Jacobian elliptic fibration determines a primitive lattice embedding $j \colon H \hookrightarrow S $ where $H$ is  spanned inside $S$ by the classes of $f$ and $\sigma$. This embedding determines, in turn, an orthogonal decomposition 
\[
S \simeq H \oplus K(-1),
\]
where $K$ is an even positive-definite lattice of rank $\rho_S-2$.  The lattice $K$ is usually referred to as the \emph{frame} of the fibration $\pi$ and carries important geometric information about $\pi$. The \emph{root sub-lattice} $K^{\mathrm{root}} \subset K$ decomposes canonically as a sum of ADE-type lattices, and hence it encodes information about the singular fibers of the fibration $\pi$. Second, a theorem of Shioda proved that the factor group $\mathcal{W} = K / K^{\mathrm{root}}$ is canonically isomorphic to the Mordell--Weil group $\mathcal{W} \simeq \operatorname{MW}(\mathcal{X},\pi)$, providing information about the sections of $\pi$.  
\par The arguments above may also be reversed, see \cite{MR2369941}: a primitive lattice embedding $ H \hookrightarrow S $ is induced by a Jacobian elliptic fibration if and only if the image of $H$ in $S$ contains a quasi-ample class. Furthermore, one can prove that there is a one-to-one correspondence between Jacobian elliptic fibrations on $X$, up to action of the group of automorphisms $\operatorname{Aut}(X)$, and the primitive lattice embeddings $H \hookrightarrow S$, up to the action of the isometry group $\mathrm{O}(S)$.   
\par A related question in the above context is the following: given a fixed choice of frame, how many non-isomorphic Jacobian elliptic fibrations are there on $\mathcal{X}$ with orthogonal complement $K(-1) = j(H)^\perp$. Festi and Veniani \cite{FestiVeniani20} have proved that the answer is always a finite number $\operatorname{mult}(K)$, referred to as the \emph{multiplicity} of the frame. The number can be computed lattice-theoretically, by counting the elements of a particular coset space of isometries. To describe this space, note that one has the following group homomorphisms:
\begin{equation}
\label{eqn:mor}
\mathrm{O}_h(T_{\mathcal{X}}) \hookrightarrow \mathrm{O}(T_{\mathcal{X}})  \rightarrow  \mathrm{O}\left( Q_{T_{\mathcal{X}}} \right), \qquad
\mathrm{O}(S) \rightarrow \mathrm{O}\left( Q_S \right).
\end{equation} 
Here, the groups on the right are the  automorphism groups preserving the associated discriminant quadratic forms, and $\mathrm{O}_h(T_{\mathcal{X}})$ is the subgroup of isometries in $\mathrm{O}(T_{\mathcal{X}})$ that preserve the Hodge decomposition. Moreover, due to the fact that $T_{\mathcal{X}}$ is the orthogonal complement of $S \simeq H \oplus K(-1)$
within an overall unimodular lattice, one has (non-canonical) isomorphisms:
\[
D(T_\mathcal{X}) \ \simeq \ D(S) \ \simeq \ D(K) ,
\]
with the discriminant quadratic forms on $D(T_\mathcal{X})$ and $D(K)$ only differing by a minus sign. Denote by $G_1$ and $G_2$ the abelian groups representing the images of the morphisms~$\eqref{eqn:mor}$. Up to an overall conjugation, one may then regard $G_1$ and $G_2$ as subgroups within the isometry group $\mathrm{O}(Q_{K})$, i.e., the orthogonal group of the discriminant form
on $D(K)$. The double coset space:
\begin{equation}
G_1 \backslash \mathrm{O}\left(Q_{K}\right)  / G_2 
\end{equation}   
is then well-defined, and Festi and Veniani \cite{FestiVeniani20}*{Thm~2.8} proved
\[
\operatorname{mult}(K) \ = \ \Big\vert \, G_1 \backslash \mathrm{O}\left(Q_{K}\right)  / G_2 \,\Big\vert .
\] 
 In particular, the index can be computed using the positive definite lattice $K$.  In geometric terms, this quantity determines the number of Jacobian elliptic fibrations with given singular fibers and Mordell-Weil group, up to automorphisms.  
\par By \cite{MR1892313}, one has $\mathrm{O}_h(T_{\mathcal{X}}) = \{\pm \mathrm{id}\}$ whenever the Picard number \(\rho_S\) is odd, and also for a very general K3 surface when \(\rho_S\) is even. In this situation, the image \(G_1\) of \(\mathrm{O}_h(T_{\mathcal X})\) in \(\mathrm{O}(Q_K)\) is trivial or consists only of the classes induced by \(\pm\mathrm{id}\), and the double-coset count simplifies considerably. We identify
\[
G_2 \simeq \operatorname{Im}\!\left(\mathrm{O}(K)\rightarrow \mathrm{O}(Q_K)\right)
\]
with the image of the natural homomorphism induced by the action of lattice isometries on the discriminant form. Then the multiplicity is determined entirely by the failure of isometries of the discriminant form to lift to isometries of \(K\). In particular, when the action of \(G_1\) is trivial, one obtains $\operatorname{mult}(K) =[\mathrm{O}(Q_K):G_2]$. Since \(\mathrm{O}(Q_K)\) is finite, Lagrange's theorem yields
\begin{equation}
\label{eqn:mult}
\operatorname{mult}(K)
=
\bigl[\mathrm{O}(Q_K):\operatorname{Im}(\mathrm{O}(K))\bigr]
=
\frac{|\mathrm{O}(Q_K)|}
     {|\operatorname{Im}(\mathrm{O}(K))|}.
\end{equation}
Equivalently, the multiplicity is the cardinality of the quotient of \(\mathrm{O}(Q_K)\) by the subgroup of automorphisms of the discriminant form induced by lattice isometries of \(K\).  Geometrically, a multiplicity greater than one reflects the existence of automorphisms of the discriminant form that do not arise from isometries of \(K\), and hence correspond to distinct Jacobian elliptic fibrations having the same frame.
\subsection{The lattices of the Vinberg family}
We consider the families of lattice-polarized K3 surfaces, whose very general member has the transcendental lattice $T_{\mathcal{X}}$ given by
\[
T_n \;\simeq\; H \oplus H \oplus D_n(-1), \qquad  4 \leq n \leq 12
\]
and $H \oplus H \oplus A_3(-1)$ and $H \oplus H \oplus  A_1(-1)^{\oplus 2}$ for $n=3, 2$. The orthogonal complement of $T_n$ in the K3 lattice $\Lambda_{K3}$ is isometric to $S_n =  H \oplus D_{16-n}(-1)$ for $2 \leq n \leq 12$ and has rank $\rho_{S_n}=18-n$ with $6 \le \rho_{S_n} \le 16$.  The lattice $S_n$ is of finite automorphism type if $\rho_{S_n}  \neq 11, 15$ and is 2-elementary if $\rho_{S_n}$ is even; see \cite{MR3165023}. As there is a primitive embedding $H \hookrightarrow S_n$, a K3 surface $\mathcal{X}$ with N\'eron--Severi lattice $ \mathrm{NS}(\mathcal{X}) \cong S_n$ always admits a Jacobian elliptic fibration with a single reducible fiber of type $D_{16-n}$. 
\par One checks the following:
\begin{lemma}
\label{lem:multiplicity0}
For a very general $S_n$-polarized K3 surface $\mathcal{X}$, the Jacobian elliptic fibration with a single reducible fiber of type $K_n=D_{16-n}$ is unique, except in the case $\rho_{S_n}=14$ $(n=6)$ when the multiplicity is $\operatorname{mult}(K_n)=3$.
\end{lemma}
\begin{proof}
The proof follows by a computation of the multiplicities, given by Equation~\eqref{eqn:mult}.
\end{proof}

Roulleau determined for very general K3 surfaces of finite automorphism type the number of all $(-2)$-curves and their dual graphs; see \cite{MR4526267}.  We make the following observation:
\begin{remark}
 For $\rho_{S_n}=6, \dots, 10$ (or $n=12, \dots,8$) the dual graph of rational curves for a very general $S_n$-polarized K3 surface consists exactly of the classes that form 
 an extended Dynkin diagram $\widetilde{D}_{16-n}$ and one additional $(-2)$-class representing the zero section.
\end{remark}
For $\rho_{S_n}=12, 13, 14, 16$ (or $n=6,5,4,2$) there are additional $(-2)$-curves in the dual graph. One example is the dual graph in Figure~\ref{fig:pic14vin} where an additional node is present. 

\par We will restrict our attention to $S_n$-polarized K3 surfaces of Picard number $\rho_{S_n} \geq 10$.  This assumption is motivated by general constraints arising from the theory  of arithmetic quotients of type IV domains. Indeed, by a result of  Shvartsman and Vinberg \cite{MR3647781}, a necessary condition for  the freeness of the algebra of automorphic forms of a non-cocompact  arithmetic group $\Delta$ acting on a symmetric domain of type IV of dimension $d$ is that $d \leq 10$.  Since the dimension of the period domain is $20 - \rho_{S_n}$,  this condition translates into $\rho_{S_n} \geq 10$. 

\par We call the family of lattice polarized K3 surfaces with polarizing lattice $S_n$ of rank $10 \le \rho_{S_n} \le 16$ the \emph{Vinberg family of rank $\rho_{S_n}$}.  The lattice-theoretic data of the Vinberg family is summarized in Table~\ref{tab:K3JEF_b}: for each Picard number $\rho_{S_n}$, we list the transcendental lattice $T_n$ of a very general K3 surface polarized by the lattice $S_n$, the determinant and parity of its discriminant form (if the lattice is 2-elementary), the possible decomposition(s) of the N\'eron--Severi lattice $\mathrm{NS}(\mathcal{X})$ of the form $H \oplus K(-1)$, the data $(K^{\mathrm{root}}, \mathcal{W})$ for the corresponding Jacobian elliptic fibration, with the root sublattice $K^{\mathrm{root}}$ and the Mordell--Weil group $\mathcal{W} \simeq \operatorname{MW}(\mathcal{X},\pi)$, and the multiplicity of the frame. As mentioned above, the presence of additional rational curves in the dual graph for $\rho_{S_n} \ge 11$ allows for other Jacobian elliptic fibrations  to be supported on the K3 surfaces, with the polarizing lattice admitting several isometric decompositions of the form $H \oplus K(-1)$. Whenever the polarizing lattice is of finite automorphism type in the sense of Nikulin \cite{MR3165023}, the results of \cite{MR4704757} can be used to show that Table~\ref{tab:K3JEF_b} lists \emph{all} Jacobian elliptic fibrations supported on the corresponding very general K3 surfaces. The case of $\rho_{S_n}=16$ coincides with a case of another sequence of lattice polarized K3 surfaces that was already investigated by the authors in \cite{MR4987329}. For $\rho_{S_n}= 11, 15$ the lattices are not of finite automorphism type, which we indicated by $*$. This follows immediately from the fact that the K3 surfaces also admit the indicated Jacobian elliptic fibration with Mordell--Weil group of positive rank, and we list the height-pairing of the infinite-order section generator. 

We have the following:
\begin{proposition}
\label{lem:multiplicity}
For a very general $S_n$-polarized K3 surface $\mathcal{X}$ with $4 \le n \le 8$, the frames of the Jacobian elliptic fibrations supported on $\mathcal{X}$ in Table~\ref{tab:K3JEF_b} have multiplicity one, with the exception of $\rho_{S_n}=14$ where the frame with $(K^{\mathrm{root}}, \mathcal{W})=(D_{12}, \{ \mathbb{I} \})$ has multiplicity 3.
\end{proposition}
\begin{proof}
The proof uses Lemma~\ref{lem:multiplicity0} and follows by computation of the multiplicities~\eqref{eqn:mult}.
\end{proof}

\begin{remark}
The first two of the Jacobian elliptic fibrations listed in Table~\ref{tab:K3JEF_b} for each entry or rank $\rho_{S_n} \ge 11$ will play a crucial role in the construction of the orthogonal modular forms associated with their period domains. They will be constructed in Propositions~\ref{prop:WEQ_rank10} and~\ref{prop:geometry}.
\end{remark}

\begin{table}[!ht]
\scalebox{0.9}{
\begin{tabular}{|c|cclc|cc|c|}
\hline
$\rho_{S_n} \, (n)$ & $T_n$ &  \multicolumn{1}{c}{$S_n= H \oplus K(-1)$}  & $|\det{S_n}|$ & $\delta$ & $K^{\text{root}}$ & $\mathcal{W}$ & $\operatorname{mult}(K)$\\
\hline
\hline
$10$ $(8)$	& $H \oplus H \oplus D_8(-1)$ & $H\oplus D_8(-1)$ 					& $2^2$ 		& $0$ 	& $D_8$ 			& $\{ \mathbb{I} \}$						& 1\\
\hline
$11^*(7)$ 	& $H \oplus H \oplus D_7(-1)$ & $H\oplus D_9(-1)$ 				& $2^2$ 	& -- 	& $D_9$ 		& $\{ \mathbb{I} \}$					& 1\\
& & $\quad \cong H \oplus E_8(-1) \oplus A_1(-2)$ 					& 		& 		& $E_8$ 		& $\langle 4 \rangle$					& 1\\
\hline
$12$ $(6)$ 	& $H \oplus H \oplus D_6(-1)$ & 
$H \oplus D_{10}(-1)$ 					& $2^2$	& $1$		& $D_{10}$ 		& $\{ \mathbb{I} \}$					& 1\\
& & $\quad \cong H\oplus E_8(-1) \oplus A_1(-1)^{\oplus 2}$ 					& 		& 		& $E_8 + 2 A_1$ 	& $\{ \mathbb{I} \}$					& 1\\
\hline
$13$ $(5)$ 	& $H \oplus H \oplus D_5(-1)$ & 
$H \oplus D_{11}(-1)$ 					& $2^2$	& --		& $D_{11}$ 		& $\{ \mathbb{I} \}$					& 1\\
& & $\quad \cong H\oplus E_8(-1) \oplus A_3(-1)$ 					& 		& 		& $E_8 + A_3$ 	& $\{ \mathbb{I} \}$					& 1\\
\hline
$14$ $(4)$ 	& $H \oplus H \oplus D_4(-1)$ & 
$H \oplus D_{12}(-1)$ 					& $2^2$	& $0$		& $D_{12}$ 		& $\{ \mathbb{I} \}$					& 3\\
& & $\quad \cong H\oplus E_8(-1) \oplus D_4(-1)$ 					& 		& 		& $E_8 + D_4$ 	& $\{ \mathbb{I} \}$					& 1\\
\hline
$15^*(5)$ 	& $H \oplus H \oplus A_3(-1)$ & $H\oplus D_{13}(-1)$ 				& $2^2$ 	& -- 	& $D_{12}$ 	& $\langle 1 \rangle$					& 1\\
&& $\quad \cong H\oplus E_8(-1) \oplus D_5(-1)$ 					& 		& 		& $E_8 + D_5$ 		& $\{ \mathbb{I} \}$					& 1\\ 
& &  & 		& 		& $D_{13}$ 		& $\{ \mathbb{I} \}$				& 1\\
\hline
$16$ $(4)$ & $H \oplus H \oplus 2A_1(-1)$ & $H \oplus D_{14}(-1)$ 			& $2^2$ 	& $1$ 	& $D_{12} + 2 A_1$ 	&$\mathbb{Z}/2\mathbb{Z}$ 			& 1\\
&& $\cong H \oplus E_8(-1) \oplus D_6(-1)$ 				&		&		& $E_8 + D_6$ 		& $\{ \mathbb{I} \}$ 					& 1\\
&& $\cong H \oplus E_7(-1) \oplus E_7(-1)$ 				&		&		& $2 E_7$ 		& $\{ \mathbb{I} \}$ 					& 1\\
&& 						&		&		& $D_{14}$ 		& $\{ \mathbb{I} \}$ 					& 1\\
\hline
\end{tabular}}
\captionsetup{justification=centering}
\caption{Jacobian elliptic K3 surfaces in the Vinberg family}
\label{tab:K3JEF_b}
\end{table}
\subsection{Explicit models for the Vinberg family}
We now derive a Weierstrass form for the K3 surfaces in the Vinberg family. We have the following:
\begin{proposition}
\label{prop:WEQ_rank10}
Consider a Jacobian elliptic K3 surface with Weierstrass equation
\begin{equation} 
\label{eqn:WEQ}
\begin{split}
 y^2  = x^3 + & \Big( u^3 + F_4 u + F_6\Big) x^2 + \Big( G_8 u^2 + G_{10} u + G_{12} \Big) x \\
 + & \ H_8 u^5 + H_{10} u^4 + H_{12} u^3 + H_{14} u^2 + H_{16} u + H_{18}.
\end{split}
\end{equation}
The minimal resolution of Equation~(\ref{eqn:WEQ}) defines a K3 surface endowed with a canonical $H\oplus D_8(-1)$-polarization. Conversely, every $H\oplus D_8(-1)$-polarized  K3 surface has a birational model given by Equation~(\ref{eqn:WEQ}).
\end{proposition}
\begin{proof}
A K3 surface with N\'eron--Severi lattice $H \oplus D_{16-n}(-1)$ admits an elliptic fibration that generically has a singular fiber at $u=\infty$ of type $I_{12-n}^*$ and $(n+6)$ singular fibers of type $I_1$. A Weierstrass model is then obtained by moving the reducible singular fiber to $u=\infty$. We write the Weierstrass model for $n \le 8$ in the form
\begin{equation}
\begin{split}
 y^2   = x^3 & + \Big( \tilde{F}_0 u^3 + \tilde{F}_2 u^2 + F_4 u + F_6\Big) x^2 + \Big( \tilde{G}_4 u^4 + \tilde{G}_6 u^3 +G_8 u^2 + G_{10} u + G_{12} \Big) x \\
 & +  \ H_8 u^5 + H_{10} u^4 + H_{12} u^3 + H_{14} u^2 + H_{16} u + H_{18}.
\end{split}
\end{equation}
We can then normalize the coefficient of $u^3x^2$ to one, and eliminate the subleading term $u^2x^2$, setting $\tilde{F}_0=1$ and $\tilde{F}_2=0$. The reason is that $\tilde{F}_0=0$ would yield a non-minimal Weierstrass model with a $(4,6,12)$-point at $u=\infty$ (for a definition see \cite{MR1078016}). Overall shifts in $x$ are then uniquely constrained by requiring that the term in $x$ has minimal degree, yielding $\tilde{G}_4=\tilde{G}_6=0$. It follows from Lemma~\ref{lem:multiplicity} that the Weierstrass model is unique, up to rescalings $x \mapsto \lambda^6 x$, $y\mapsto \lambda^9 y$, $u \mapsto \lambda^2 u$ for $\lambda \in \mathbb{C}^\times$. Under this action, the coefficients rescale according to $F_{2k} \mapsto \lambda^{2k} F_{2k}$, $G_{2l} \mapsto \lambda^{2l} G_{2l}$, and $H_{2m} \mapsto \lambda^{2m} H_{2m}$. This means we have coefficients of weights
\begin{equation}
\label{eqn:weights}
 4, 6, 8, 8, 10, 10, 12, 12, 14, 16, 18.
\end{equation}
Moreover, one can then tell precisely when two members of the family in Equation~(\ref{eqn:WEQ}) are isomorphic:  they are isomorphic if and only if their coefficients are related by a rescaling.
\end{proof}
\par We make the following:
\begin{remark}
\label{rem:11-14}
 In Table~\ref{tab:K3JEF_b} restrictions to higher Picard numbers $\rho_{S_n} =11, \dots, 14$ are attained by setting $H_8=0$, $H_8=H_{10}=0$, $\dots$, $H_8= \dots = H_{16}=0$.  
\end{remark}
\begin{remark}
\label{rem:15}
 For $\rho_{S_n}=15$ we have $H_8= \dots = H_{16}=0$. Equation~(\ref{eqn:WEQ}) then has a singular fiber at $u=\infty$ of type $I_{8}^*$. If we write $H_{18}=H_9^2$, Equation~\eqref{eqn:WEQ} admits the infinite-order section $\sigma'\colon (x,y)=(0,\pm H_9)$. The N\'eron--Severi lattice is the extension of $H \oplus D_{12}(-1)$ by the infinite-order section $\sigma'$.  In fact, the N\'eron--Severi lattice is generated as
\[
 \Big\langle \sigma + f, f , \ d_1, d_2, \dots, d_{12}, \ \sigma' - \sigma - 2f \Big \rangle,
\]
where $d_1, \dots, d_{12}$ are the nodes of the $D_{12}$-graph, with $d_1 \circ \sigma'=1$ and $d_i \circ \sigma'=0$ for $i=2, \dots, 12$, and $\sigma' \circ f=1$, $\sigma' \circ \sigma =0$. In the reducible fiber, the section $\sigma'$ intersects a terminal node connected with the fork opposite to one containing the neutral component in the extended Dynkin diagram $\widetilde{D}_{12}$ whence $\langle \sigma', \sigma' \rangle = 1$. The N\'eron--Severi lattice has rank $15$ and discriminant $4$. For the given choice of generators the lattice decomposes as $H \oplus K(-1)$. A computation of the lattice-theoretic genus shows $K \cong D_{13}$.
\end{remark}
\begin{remark}
\label{rem:16}
For $\rho_{S_n}=16$ we have $H_8= \dots = H_{18}=0$. Equation~\eqref{eqn:WEQ} then admits the 2-torsion section $(x,y)=(0,0)$; all supported Jacobian elliptic fibrations in this case were explicitly constructed in \cite{MR4903806}. 
\end{remark}
Next, let us summarize some geometric realizations of the lattice-polarized K3 surfaces listed in Table~\ref{tab:K3JEF_b}, following the constructions of \cite{MR4526267}. We focus on double-plane models and their relation to elliptic fibrations, as well as models given by double covers of Hirzebruch surfaces: Roulleau finds that for the very general $H \oplus D_{16-n}(-1)$-polarized K3 surface with $n \le 9$, there exists a divisor $D_2$ with $D_2^2 = 2$ inducing a double cover
\[
\varphi_{|D_2|} : \mathcal{X} \longrightarrow \mathbb{P}^2,
\]
so that the K3 surface is obtained as the minimal resolution of a double cover of the projective plane branched along a sextic curve.  Second, he shows that there exists a hyperelliptic divisor giving rise to a morphism
\[
\mathcal{X} \longrightarrow \mathbb{F}_n,
\]
typically with $n=4$, whose branch locus is of the form $\sigma + B$ with $B \in |3\sigma + 12f|$.  We now give a closely related characterization of the K3 surfaces of the Vinberg family. Our models are derived explicitly from the Weierstrass equation~\eqref{eqn:WEQ}.
\begin{proposition}
\label{prop:geometry}
\begin{enumerate}[
    label=(\roman*),
    wide,
    labelindent=0pt,
    itemsep=4pt,
    topsep=0pt,
    parsep=0pt,
    partopsep=0pt
]
\item Equation~\eqref{eqn:WEQ} realizes the K3 surface as a double sextic over $\mathbb{P}^2$. For $\rho_{S_n} = 10$, the sextic branch curve decomposes into the union of a line and a quintic, and the double-sextic has singularities $\mathbf{E}_7 + 2 \mathbf{A}_1$.
\item \label{prop:geometry.partII}
For rank $\rho_{S_n} \geq 11$, Equation~\eqref{eqn:WEQ} realizes the K3 surface as a double cover of a quadric surface in $\mathbb{P}^1 \times \mathbb{P}^1$. The two rulings induce the first two of the Jacobian elliptic fibrations listed in Table~\ref{tab:K3JEF_b} for each entry.
\item For rank $\rho_{S_n} \geq 13$, Equation~\eqref{eqn:WEQ} realizes the K3 surface as a quartic hypersurface in $\mathbb{P}^3$. For $\rho_{S_n} = 13$, the quartic hypersurface has a single $\mathbf{A}_{11}$ singularity.
\end{enumerate}
\end{proposition}
\begin{proof}
\emph{Rank $\rho_{S_n} \ge 10$:}
From Equation~(\ref{eqn:WEQ}), by setting $x=z_1/z_3$, $u=z_2/z_3$, $y=\tilde{z}_4/z_3^3$ we obtain a double sextic surface over $\mathbb{P}^2 = \mathbb{P}(z_1, z_2,  z_3)$, given by
\begin{equation}
 \begin{split}
  \tilde{z}_4^2  = z_3 & \Big( z_1^3 z_3^2 + \Big(  z_2^3 + F_4 z_2 z_3^2 + F_6 z_3^3 \Big) z_1^2 + \Big( G_8 z_2^2 + G_{10} z_2 z_3 + G_{12} z_3^2 \Big) z_1 z_3^2 \\
 & +  \ H_8 z_2^5  + H_{10} z_2^4 z_3 + H_{12} z_2^3 z_3^2 + H_{14} z_2^2 z_3^3 + H_{16} z_2 z_3^4 + H_{18} z_3^5\Big).
\end{split}
\end{equation}
The branch sextic has only normal singularities. In rank 10, the branch curve exhibits one singularity of type $\mathbf{E}_7$ at $[z_1, z_2,  z_3, \tilde{z}_4]=[1: 0 : 0: 0]$ and two singularities of type $\mathbf{A}_1$ at $[z_1, z_2,  z_3, \tilde{z}_4]=[1 : \alpha : 0: 0]$ with $\alpha^2 G_8+1=0$. 

\medskip

\emph{Rank $\rho_{S_n} \ge 11$:}
For the case with $H_8 = 0$, Equation~(\ref{eqn:WEQ}) defines a double cover of $\mathbb{P}(x_0, x_1) \times \mathbb{P}(u_0, u_1)$, given by
\begin{equation} 
\begin{split}
  \tilde{y}^2  =   x_1 \Big( & u_1^4 x_0^3 + u_1 x_1 x_0^2\Big( u_0^3 + F_4 u_0 u_1^2 + F_6 u_1^3\Big)  + u_1^2 x_1^2 x_0  \Big( G_8 u_0^2 + G_{10} u_0 u_1 + G_{12} u_1^2 \Big) \\
 & + x_1^3 \Big( H_{10} u_0^4 + H_{12} u_0^3 u_1 + H_{14} u_0^2 u_1^2 + H_{16} u_0 u_1^3 + H_{18} u_1^4 \Big) \Big).
\end{split}
\end{equation}
The two rulings of $\mathbb{P}^1 \times \mathbb{P}^1$ give rise to distinct genus-one fibrations; in the cases under consideration, one verifies that both induce elliptic fibrations with section.

\medskip

\emph{Rank $\rho_{S_n} \ge 13$:}
Assuming $H_8 = H_{10} = H_{12}=0$ and writing $H_{14}=H_7^2$, one obtains a quartic surface in $\mathbb{P}^3 = \mathbb{P}(\mathbf{x}_0, \mathbf{x}_1, \mathbf{x}_2, \mathbf{x}_3)$ by the homogeneous quartic equation
\begin{equation}
\label{eqn:Vinberg_body1}
  \mathbf{x}_0^2 \mathbf{x}_2 \mathbf{x}_3 - 4 \mathbf{x}_1^3 \mathbf{x}_3 - \mathbf{x}_2^4 - \mathbf{x}_1 \mathbf{x}_3^2 \, g(\mathbf{x}_0, \mathbf{x}_1, \mathbf{x}_3) - \mathbf{x}_2 \mathbf{x}_3 \, f(\mathbf{x}_1, \mathbf{x}_2, \mathbf{x}_3) = 0 \,,
\end{equation}
with
\begin{equation}
\begin{split}
  g &  = 16 H_7 \mathbf{x}_0 + 16 G_8 \mathbf{x}_1 + 64 H_{16} \mathbf{x}_3, \\
  f & = 4 F_4 \mathbf{x}_1 \mathbf{x}_2 + 4 F_6 \mathbf{x}_2^2  + 16 G_{10} \mathbf{x}_1 \mathbf{x}_3  + 16 G_{12} \mathbf{x}_2 \mathbf{x}_3 + 64 H_{18} \mathbf{x}_3^2.
\end{split}
\end{equation}  
Since $A, B, C \neq 0$ in \cite{Vinberg2010}*{Eqn.~\!(13)}, we can rescale the coordinates to achieve $A=-1, B=4, C=1$ and obtain Equation~(\ref{eqn:Vinberg_body1}). Moreover,  setting
\begin{equation}
  \mathbf{x}_0 = 8 y + 8 H_7 u, \quad \mathbf{x}_1 = 4 u x, \quad \mathbf{x}_2 = 4 x, \quad \mathbf{x}_3=1,
\end{equation}  
in Equation~(\ref{eqn:Vinberg_body1}) we obtain Equation~(\ref{eqn:WEQ}) with $H_8 = H_{10} = H_{12}=0$. A computation shows that the hypersurface exhibits an $\mathbf{A}_{11}$ singularity at the point $(1:0:0:0)$.
\end{proof}
We use Proposition~\ref{prop:geometry}~\ref{prop:geometry.partII} to construct the unique fibration for the Vinberg family of rank $\rho_{S_n}=14$ in Table~\ref{tab:K3JEF_b}. After setting $H_8= \dots = H_{14} =0$ in Equation~\eqref{eqn:WEQ} and switching to the second ruling induced by the projection onto $\mathbb{P}^1_{(x)}$, the Jacobian elliptic fibration has the singular fibers $I_0^* + II^* + 8 I_1$, i.e.\! the frame $(E_8 + D_4, \{ \mathbb{I} \})$, and is given by the Weierstrass equation
\begin{equation}
\label{eqn:weq_alt}
\begin{aligned}
\eta^2
&=
\xi^3
+
x^2
\left(
F_4 x^2 
+
G_{10} x 
+
\tilde{H}_{16}
\right) \xi
+
x^3
\left(
x^4
+
F_6 x^3 
+
\tilde{G}_{12} x^2 
+
\tilde{H}_{18} x 
+
\tilde{H}_{24}
\right),
\end{aligned}
\end{equation}
where we have set
\begin{equation}
\label{eqn:relats_coeffs}
\begin{aligned}
\tilde{G}_{12}
&=
-\frac{1}{3}F_4 G_8 + G_{12},
&
\qquad \tilde{H}_{16}
&=
-\frac{1}{3}G_8^2 + H_{16},
\\[6pt]
\tilde{H}_{18}
&=
-\frac{1}{3}G_8 G_{10} + H_{18},
&
\qquad \tilde{H}_{24}
&=
\frac{1}{27}G_8\left(2G_8^2 - 9H_{16}\right),
\end{aligned}
\end{equation}
and $(\xi, \eta)$ are the fiber coordinates and $x$ is the affine coordinate on the base curve. Let us now explain that this implies that the multiplicity of the other fibration in Table~\ref{tab:K3JEF_b} for $\rho_{S_n}=14$ equals three:  the $I_0^*$-fiber in \eqref{eqn:weq_alt} is located at $x=0$. After undoing the quadratic twist, the corresponding $I_0$-fiber at $x=0$ is given by 
\begin{equation}
\label{eqn:I0fiber}
\eta^2 = \xi^3 + \xi \tilde{H}_{16} + \tilde{H}_{24}.
\end{equation}
Let the three non-trivial $2$-torsion points of the $I_0$-fiber be given by $(\xi,\eta) = \big(\tfrac{1}{3}(\psi_1-\psi_2),0\big)$, $\big(\tfrac{1}{3}(\psi_2-\psi_3),0\big)$, $\big(\tfrac{1}{3}(\psi_3-\psi_1),0\big)$ with $\psi_1 + \psi_2 + \psi_3 = 0$, so that the group $\mathrm{O}(D_4^{\vee}/D_4)$ can naturally be identified with the permutations of $\{ \psi_1, \psi_2, \psi_3\}$. In turn, they determine the coefficients $\tilde{H}_{16}$ and $\tilde{H}_{24}$ in Equation~\eqref{eqn:I0fiber}. However, Equations~\eqref{eqn:relats_coeffs} now turn out to have \emph{three} distinct solutions, namely
\begin{equation}
\begin{aligned}
G_8 &= \psi_1-\psi_2, 
\quad H_{16} = -\psi_1\psi_2, \\[6pt]
G_8 &= \psi_3 - \psi_1, 
\quad H_{16} =  - \psi_1 \psi_3, \\[6pt]
G_8 &= \psi_2 -\psi_3, 
\quad H_{16} = -\psi_2 \psi_3.
\end{aligned}
\end{equation}
Each solution defines, by means of Equation~\eqref{eqn:WEQ}, a Jacobian elliptic fibration with singular fibers $I_8^* + 10 I_1$, but for different moduli. We will determine the modular forms representing $\{ \psi_1, \psi_2, \psi_3\}$ in Section~\ref{ssec:4}. We have the following: 
\begin{remark}
 For $\rho_{S_n}=14$ the multiplicity of the frame with $(K^{\mathrm{root}}, \mathcal{W})=(D_{12}, \{ \mathbb{I} \})$ equals three. This is due to the existence of three distinct primitive embeddings $H \hookrightarrow S_{14}$ leading to non-isomorphic Jacobian elliptic fibrations with the same frame data. It is based on choosing 2 out of 3 rational curves in Figure~\ref{fig:pic14vin_fib2} as part of the reducible fiber $\widetilde{D}_{12}$.
\end{remark}
\subsection{Moduli spaces and unirationality}
The subvariety of parameter points where Equation~(\ref{eqn:WEQ}) fails to define a K3 surface consists of two components, which we shall denote by $\mathcal{C}_1$ and $\mathcal{C}_2$. Both components determine curves in the 10-dimensional weighted projective space $\mathbb{P}(4,6,8,8,10,10,12,12,14,16,18)$ which, by a slight abuse of notation we shall also denote by $\mathcal{C}_1$, $\mathcal{C}_2$.   
\par The first component $\mathcal{C}_1 \cong \mathbb{P}(4,6)$ is cut out by
\begin{equation}
\label{eqn:component_1}
 \mathcal{C}_1 \colon \quad ( G_8, G_{10}, G_{12}, H_8, \dots, H_{18} ) = 0.
\end{equation}
On this component, Equation~(\ref{eqn:WEQ}) carries a one-dimensional singular locus. In fact, one can check via a direct computation that:
\begin{lemma}
The discriminant $\Delta$ of the Weierstrass equation~(\ref{eqn:WEQ}) vanishes identically if and only if condition \eqref{eqn:component_1} holds.   
\end{lemma}
Geometrically, parameters in $\mathcal{C}_1$ correspond to a special class of Type II degenerations in the sense of Kulikov--Persson--Pinkham; see \cites{MR506296,MR506295,MR690262}:
\begin{remark}
Parameter points in $\mathcal{C}_1$ correspond, after semistable reduction,
to the situation in which every fiber in Equation~(\ref{eqn:WEQ}) degenerates to an $I_2$-fiber and the K3 surface degenerates to a union of ruled rational surfaces meeting along a common elliptic curve, with the elliptic curve appearing as a bi-section in each ruling. 
\end{remark}
\par The second component corresponds to the $(4,6,12)$-points of the Weierstrass equation~(\ref{eqn:WEQ}) which we can rewrite in the form 
\begin{equation}
 y^2 = x^3 + f(u) x + g(u).
\end{equation} 
Recall that a \((4,6,12)\)-point, as defined in \cite{MR1078016}, is a parameter point where:
\[
\operatorname{ord}(f) \geq 4, \qquad
\operatorname{ord}(g) \geq 6, \qquad
\operatorname{ord}(\Delta) \geq 12,
\]
with $\Delta= 4f^3+27g^2$.
This is the situation when, generically, the Equation~(\ref{eqn:WEQ}) carries a singularity of elliptic type. On this component, the Weierstrass coefficient polynomials $f$ and $g$ are of the form:

\begin{equation}
\begin{split}
 f(u) & = - \frac{1}{3} (u + \alpha)^4 \left(u^2 - 4 \alpha u + 10 \alpha^2 + 2 F_4\right) , \\
 g(u) & =\frac{2}{27}(u+\alpha)^6\left(u^3-6\alpha u^2+(21\alpha^2+3F_4)u-56\alpha^3-18\alpha F_4+3F_6\right).
\end{split}
\end{equation}
A direct computation leads to the following locus, parametrized by
$[\alpha:F_4]\in\mathbb P(2,4)$:
\begin{equation}
\label{eqn:curve_componentII}
\begin{aligned}
\begin{pmatrix}
F_6\\
G_8\\
G_{10}\\
G_{12}
\end{pmatrix}
&=
\begin{pmatrix}
10\alpha^3\\
-15\alpha^4\\
-12\alpha^5\\
30\alpha^6
\end{pmatrix}
+
F_4
\begin{pmatrix}
4\alpha\\
-4\alpha^2\\
4\alpha^3\\
26\alpha^4
\end{pmatrix}
+
F_4^2
\begin{pmatrix}
0\\
\frac{1}{3}\\[0.2em]
\frac{8}{3}\alpha\\[0.2em]
\frac{16}{3}\alpha^2
\end{pmatrix}
\\[1em]
\begin{pmatrix}
H_8\\
H_{10}\\
H_{12}\\
H_{14}\\
H_{16}\\
H_{18}
\end{pmatrix}
&=
\begin{pmatrix}
-\alpha^4\\
-8\alpha^5\\
-30\alpha^6\\
-70\alpha^7\\
-50\alpha^8\\
24\alpha^9
\end{pmatrix}
+
F_4
\begin{pmatrix}
-\frac{2}{3}\alpha^2\\[0.2em]
-\frac{16}{3}\alpha^3\\[0.2em]
-\frac{59}{3}\alpha^4\\[0.2em]
-\frac{128}{3}\alpha^5\\[0.2em]
-\frac{52}{3}\alpha^6\\[0.2em]
\frac{112}{3}\alpha^7
\end{pmatrix}
+
F_4^2
\begin{pmatrix}
-\frac{1}{9}\\[0.2em]
-\frac{8}{9}\alpha\\[0.2em]
-\frac{28}{9}\alpha^2\\[0.2em]
-\frac{46}{9}\alpha^3\\[0.2em]
\frac{46}{9}\alpha^4\\[0.2em]
\frac{152}{9}\alpha^5
\end{pmatrix}
+
F_4^3
\begin{pmatrix}
0\\
0\\
\frac{1}{27}\\[0.2em]
\frac{4}{9}\alpha\\[0.2em]
\frac{16}{9}\alpha^2\\[0.2em]
\frac{64}{27}\alpha^3
\end{pmatrix}.
\end{aligned}
\end{equation}
Also, a Gr\"obner basis computation yields the following:
\begin{lemma}
\label{lemma:ideal}
The ideal of the image curve~\eqref{eqn:curve_componentII} is minimally generated by
17 weighted-homogeneous generators of weights
\[
12,12,14,14,16,16,16,16,16,18,18,18,18,18,20,20,20 . 
\]
The precise polynomial expressions are given in Appendix~\ref{sec:ideal}, Equations~\eqref{eqn:weight12_14}--\eqref{eqn:weight20}. 
\end{lemma}
We denote by $\mathcal{C}_2$ the image of the parametrization~\eqref{eqn:curve_componentII}, equivalently the curve cut out by the equations of Lemma~\ref{lemma:ideal}.
\begin{remark}
In terms of Type II semistable degenerations (see \cites{MR1078016,MR4486788}), the parameter points in component $\mathcal{C}_2$ correspond to the situation when the Weierstrass fibration~(\ref{eqn:WEQ}) remains an elliptic fibration but the total space acquires an elliptic singularity. After performing a resolution of singularities, one obtains a union of two elliptic rational surfaces meeting along an elliptic curve, with the common elliptic curve appearing as a fiber of the elliptic fibration on each component.
\end{remark}
Recall then (see \cite{MR2355598}) that, on a K3 surface, Jacobian elliptic fibrations with a single reducible $D_8$-fiber (and necessarily trivial Mordell--Weil group, see \cite{MR1813537}) are equivalent to quasi-ample $H \oplus D_8(-1)$ polarizations. Connecting the above with the proof of Proposition \ref{prop:WEQ_rank10}, one then obtains:  
\begin{proposition}
The moduli space $\mathscr M_S$ of
$S=H\oplus D_8(-1)$-polarized K3 surfaces is identified with
\[
\mathbb{P}(4,6,8,8,10,10,12,12,14,16,18)
\setminus
(\mathcal{C}_1\cup\mathcal{C}_2).
\]
In particular, this space $\mathscr{M}_S$ is unirational.
\end{proposition}
The curves $\mathcal{C}_1$ and $\mathcal{C}_2$ therefore constitute the
discriminant locus in the weighted-projective parameter space.
\par Versions of the above result also hold for the remaining lattice polarizations of interest in this paper. Note that, by Vinberg's work \cite{Vinberg2010}, \cite{Vinberg2018}, the graded rings of modular forms for $\mathrm{O}^+(H^{\oplus 2} \oplus D_n(-1))$, $n \le 8$ are known to be freely generated. Wang and Williams \cite{WW21} have also shown that the graded rings of meromorphic modular forms for $\mathrm{O}^+(H^{\oplus 2} \oplus D_n(-1))$, $9 \le n \le 11$, with poles allowed along certain special hyperplane arrangements, are also freely generated. Following the considerations above, one obtains moduli spaces for the corresponding lattice-polarized K3 surfaces, in Picard number $7$ or higher, realized as open subsets of weighted projective spaces. These moduli spaces are therefore unirational. 
\section{Jacobi forms of lattice index and orthogonal modular forms}

\subsection{Jacobi forms of lattice index}

Jacobi forms of lattice index are a natural generalization of Jacobi forms as defined by Eichler--Zagier \cite{EZ85}. Suppose $L$ is a positive-definite even lattice, i.e. a free $\mathbb{Z}$-module of finite rank equipped with a quadratic form $Q : L \rightarrow \mathbb{Z}$ such that the bilinear form $$x \cdot y := Q(x+y) - Q(x) - Q(y)$$ is positive definite. Let $L_{\mathbb{C}} = L \otimes \mathbb{C}$. A Jacobi form (without holomorphy at $\infty$) of index $L$ is a holomorphic function $f : \mathbb{H} \times L_{\mathbb{C}} \rightarrow \mathbb{C}$ that satisfies \begin{equation}\label{eq:jacobi_modular} f \Big( \frac{a \tau + b}{c \tau + d}, \frac{z}{c \tau + d} \Big) = (c \tau + d)^k e^{2\pi i c Q(z) / (c \tau + d)} f(\tau, z), \quad \begin{pmatrix} a & b \\ c & d \end{pmatrix} \in \mathrm{SL}_2(\mathbb{Z}) \end{equation} and \begin{equation} \label{eq:quasiperiod} f(\tau, z + \lambda \tau + \mu) = e^{-2\pi i Q(\lambda) \tau - 2\pi i (\lambda \cdot z)} f(\tau, z), \quad \lambda, \mu \in L, \end{equation} where $\tau \in \mathbb{H}$ and $z \in L_{\mathbb{C}}$. The space of Jacobi forms of index $L$ and weight $k$ will be denoted $J_{k,L}$.
A Jacobi form of scalar index $m \in \mathbb{N}$ is then a Jacobi form of index $(\mathbb{Z}, Q(x) = mx^2)$. \\
Equations (\ref{eq:jacobi_modular}) and (\ref{eq:quasiperiod}) imply that $f$ has a Fourier series $$f(\tau, z) = \sum_{n = -\infty}^{\infty} \sum_{\ell \in L^{\vee}} c(n, \ell) q^n \zeta^{\ell}, \quad q = e^{2\pi i \tau}, \; \zeta^{\ell} := e^{2\pi i (\ell \cdot z)}.$$ $f$ is called \emph{weakly holomorphic} if $c(n, \ell) = 0$ for all sufficiently small $n$, and is called \emph{weak} if $c(n,\ell) = 0$ whenever $n < 0$. \\

A basis for the space of quasiperiodic functions in the sense of (\ref{eq:quasiperiod}) is given by the theta functions $$\vartheta_{L + \gamma}(\tau, z) = \sum_{x \in L + \gamma} e^{\pi i \tau (x \cdot x) + 2\pi i (x \cdot z)}, \quad \gamma \in L^{\vee}/L,$$ where $L^{\vee} = \{y \in L \otimes \mathbb{Q}: \; L \cdot y \subseteq \mathbb{Z}\}$ is the dual of $L$. It follows that $f$ admits a unique \emph{theta decomposition} $$f(\tau, z) = \sum_{\gamma \in L^{\vee}/L} f_{\gamma}(\tau) \vartheta_{L + \gamma}(\tau, z).$$ The vector $F = (f_{\gamma})_{\gamma}$ then transforms as a modular form of weight $k - \frac{1}{2} \mathrm{rank}(L)$ with respect to a representation of $\mathrm{SL}_2(\mathbb{Z})$ (if $\mathrm{rank}(L)$ is even) or its metaplectic cover $\mathrm{Mp}_2(\mathbb{Z})$ (if $\mathrm{rank}(L)$ is odd), the Weil representation attached to $L$.
The components $f_{\gamma}$ then have $q$-expansions $$f_{\gamma}(\tau) = \sum_{n \in \mathbb{Z} - Q(\gamma)} c_{\gamma}(n) q^n.$$ $f$ is called \emph{holomorphic} if $c_{\gamma}(n) = 0$ whenever $n < 0$, and is called a \emph{cusp form} if $c_{\gamma}(n) = 0$ whenever $n \le 0$. Hence one has increasingly strong adjectives
$$\left( \begin{gathered} \text{weakly holomorphic} \\ \text{Jacobi forms} \end{gathered} \right) \supseteq \left( \begin{gathered} \text{weak} \\ \text{Jacobi forms} \end{gathered}\right) \supseteq \left( \begin{gathered} \text{holomorphic} \\ \text{Jacobi forms} \end{gathered}\right) \supseteq \left( \begin{gathered} \text{Jacobi} \\ \text{cusp forms} \end{gathered}\right).$$

If $\iota : L \rightarrow M$ is an embedding of lattices then we obtain a \emph{pullback} map on Jacobi forms in the opposite direction, $$\iota^* : J_{k, M} \longrightarrow J_{k, L}, \quad \iota^* f(\tau, z) := f(\tau, \iota z)$$ which respects the notions of weakly-holomorphic, weak, holomorphic and cusp form.

\subsection{Wirthm\"uller's theorem}

The most important algebra structure result for Jacobi forms of lattice index is Wirthm\"uller's theorem. Let $R$ be an irreducible root system, viewed as a subset of Euclidean space, and let $L = L_R$ be its root lattice, i.e. the integral lattice generated by the roots in $R$ with bilinear form given by the standard Euclidean scalar product if $L$ is even or the Euclidean product rescaled by two if $L$ is odd. If $R$ is a root system of ADE type then $L$ is the corresponding ADE lattice; if $R = B_m$ then $L = m A_1$; if $R = C_m$ then $L = D_m$; if $R = F_4$ then $L = D_4$; and if $R = G_2$ then $L = A_2$. \\
In this paper, we always take the following model for the root lattices $A_n$ and $D_n$: \begin{align*} D_n &= \{x = (x_1,...,x_n) \in \mathbb{Z}^n: \; \sum x_i \in 2 \mathbb{Z}\}, \\ A_{n-1} &= \{x = (x_1,...,x_n) \in \mathbb{Z}^n: \; \sum x_i = 0\},\end{align*} with the Euclidean scalar product. The exceptional root lattices $E_6,E_7,E_8$ will not be used. \\

In all cases, the Weyl group $W = W(R)$ naturally acts on $L$, and on the rescaled lattice $L(m)$ for any $m \in \mathbb{N}$. A Jacobi form of index $L(m)$ is Weyl-invariant if it satisfies $$f(\tau, \sigma z) = f(\tau, z) \quad \text{for all} \; \sigma \in W.$$
We define the bigraded ring of Weyl-invariant weak Jacobi forms as $$J_{*, *}^{\mathrm{weak}, W_R} := \bigoplus_{k \in \mathbb{Z}} \sum_{m=0}^{\infty} J_{k, L(m)}^{\mathrm{weak}, W_R},$$ where $J_{k, L(0)}^{\mathrm{weak}, W_R} := \mathbb{C}[E_4, E_6]$ is the ring of modular forms and where $J_{k, L(m)}^{\mathrm{weak}, W_R}$ is the space of Weyl-invariant weak Jacobi forms of weight $k$ and index $L(m)$ for $m \ge 1$.

\begin{theorem}\label{theorem:wirthmueller} Suppose $R \ne E_8$. Then $J_{*, *}^{\mathrm{weak}, W_R}$ is a freely generated algebra on $E_4, E_6$ and on weak Jacobi forms in the following weights and indices:

\begin{table}[!ht]
\begin{tabular}{|l|l|}
\hline
$R$ & $(k, m)$ \\ \hline \hline
$A_n$, $n \ge 1$ & $(k, 1)$, $-(n+1) \le k \le -2$ and $k=0$ \\ \hline
$B_n$, $n \ge 2$ & $(2k, 1)$, \; $-n \le k \le 0$ \\ \hline
$C_n$, $n \ge 3$ & $(-4,1),(-2,1),(0,1)$ and $(2k, 2)$, $-n \le k \le -3$ \\ \hline
$D_n$, $n \ge 4$ & $(-n,1),(-4,1),(-2,1),(0,1)$ and $(2k,2)$, $-(n-1) \le k \le -3$ \\ \hline
$E_6$ & $(-5,1),(-2,1),(0,1),(-9,2),(-8,2),(-6,2),(-12,3)$ \\ \hline
$E_7$ & $(-2,1),(0,1),(-10,2),(-8,2),(-6,2),(-14,3),(-12,3),(-18,4)$ \\ \hline
$F_4$ & $(-2,1),(0,1),(-8,2),(-6,2),(-13,3)$ \\ \hline
$G_2$ & $(-2,1),(0,1),(-6,2)$ \\ \hline
\end{tabular}
\end{table}

\end{theorem}

In the exceptional case $R = E_8$, Wang \cite{Wang21} proved that the algebra of Weyl-invariant weak Jacobi forms is not freely generated. \\

For $R = A_n$ (and $R = B_n$), Bertola \cite{Bertola2000} constructed a natural set of generators of the algebra $J_{*, *}^{\mathrm{weak}, W_R}$. We recall Bertola's construction: \\

Let $\wp(\tau, z)$ be the Weierstrass elliptic function for the lattice $\mathbb{Z} + \tau \mathbb{Z}$, $$\wp(\tau, z) = \frac{1}{z^2} + \sum_{\substack{m, n \in \mathbb{Z} \\ (m, n) \ne (0, 0)}} \Big( \frac{1}{(z + m \tau + n)^2} - \frac{1}{(m \tau + n)^2} \Big),$$ and write $P(\tau, z) := \frac{1}{(2\pi i)^2} \wp(\tau, z)$ for the normalization with Fourier series  $$P(\tau, z) = \frac{\zeta^{-1} + 10 + \zeta}{12(\zeta^{-1} - 2 + \zeta)} + \sum_{n=1}^{\infty} \Big( \sum_{d | n} d (\zeta^{-d} - 2 + \zeta^d) \Big) q^n, \quad \mathrm{im}(\tau) > |\mathrm{im}(z)|,$$ where $q = e^{2\pi i \tau}$, $\zeta = e^{2\pi i z}$. Let $\vartheta$ be the odd Jacobi theta function \begin{align*} \vartheta(\tau, z) &= \sum_{n=-\infty}^{\infty} (-1)^n q^{(n+1/2)^2 / 2} \zeta^{n+1/2} \\ &= q^{1/8} (\zeta^{1/2} - \zeta^{-1/2}) \prod_{n=1}^{\infty} (1 - q^n) (1 - q^n \zeta) (1 - q^n \zeta^{-1}). \end{align*}
Then $\wp$ is a meromorphic Jacobi form of weight $2$ and index $0$, and $\vartheta$ is a holomorphic Jacobi form of weight $1/2$ and index $1/2$ (with appropriate multiplier systems). 
For $z = (z_1,...,z_{n+1}) \in A_n \otimes \mathbb{C}$, define the Wronskian determinants $$W^{(n; k)}(z) := \mathrm{det} \begin{pmatrix} 1 & P(z_1) & \cdots & \widehat{D^k P(z_1)} & \cdots & D^n P(z_1) \\ \cdots & \cdots & \cdots & \cdots & \cdots & \cdots \\ 1 & P(z_n) & \cdots & \widehat{D^k P(z_n)} & \cdots & D^n P(z_n) \end{pmatrix}, \quad 0 \le k \le n,$$ where $D$ is the normalized derivative $D = \frac{1}{2\pi i} \frac{\partial}{\partial z} = \zeta \frac{\partial}{\partial \zeta}$, and where the $(k+1)$st column of the matrix is omitted. Let $\Theta$ be the theta block $$\Theta(\tau, z) := \frac{\prod_{i=1}^{n+1} \vartheta(\tau, z_i)}{\Delta(\tau)^{(n+1)/4}}.$$ Finally, define $$\phi_0^{(n)}(\tau, z) := 12 \cdot \frac{(-1)^n}{n!} \Theta(\tau, z) \times \frac{W^{(n; 0)}}{W^{(n; n)}}$$ and $$\phi_{-k-1}^{(n)}(\tau, z) := \frac{(-1)^n}{n!} \Theta(\tau, z) \times \frac{W^{(n; k)}}{W^{(n; n)}}.$$

\begin{theorem} The functions $\phi_k^{(n)}$, $-(n+1) \le k \le -2$ and $k=0$, are weak Jacobi forms of index $A_n$ and weight $k$ and, together with $E_4, E_6,$ they generate the algebra of $W(A_n)$-invariant weak Jacobi forms.
\end{theorem}

There are natural embeddings $$\iota_j : A_{n-1} \rightarrow A_n, \quad (z_1,...,z_n) \mapsto (z_1,...,z_{j-1},0,z_j,...,z_n),$$ all of which induce the same pullback map on Weyl-invariant (weak) Jacobi forms, $$\iota^* = \iota_j^* : J_{k, A_n(m)}^{\mathrm{weak}, W(A_n)} \longrightarrow J_{k, A_{n-1}(m)}^{\mathrm{weak}, W(A_{n-1})}.$$ We have the following:

\begin{lemma} $\iota^* \phi_{-(n+1)}^{(n)} = 0$ and $\iota^* \phi_k^{(n)} = \phi_k^{(n-1)}$ for any $k > -(n+1)$.
\end{lemma}
\begin{proof} $\phi_{-(n+1)}^{(n)}$ is the theta block $$\phi_{-(n+1)}^{(n)} = \frac{(-1)^n}{n!} \frac{\vartheta(\tau, z_1) \times ... \times \vartheta(\tau, z_{n+1})}{\Delta^{(n+1)/4}},$$ and it vanishes when any $z_i$ is zero because $\vartheta(\tau, 0) = 0$. \\

Suppose $k > -(n+1)$. Since the Laurent expansion of $P(\tau, z)$ about $z=0$ begins $P(\tau, z) = \frac{1}{(2\pi i z)^2} + O(z^2)$, we have $$D^m P = \frac{(m + 1)! (-1)^m}{(2\pi i)^{m+2}} z^{-(m+2)} + O(z^0)$$ for every $m \ge 1$. If $k \ne -n$ we obtain, in a neighborhood of $z_n$, \begin{align*} &\quad \mathrm{det} \begin{pmatrix} 1 & P(z_1) & \cdots & \widehat{ D^kP(z_1)} & \cdots & D^n P(z_1) \\ \cdots & \cdots & \cdots & \cdots & \cdots & \cdots \\ 1 & P(z_n) & \cdots & \widehat{D^k P(z_n)} & \cdots & D^n P(z_n) \end{pmatrix} \\ &= \frac{(-1)^n (n+1)!}{(2\pi i)^{n+2} z_n^{n+1}} \mathrm{det} \begin{pmatrix} 1 & P(z_1) & \cdots & \widehat{D^k P(z_1)} & \cdots & D^{n-1} P(z_1) \\ \cdots & \cdots & \cdots & \cdots & \cdots & \cdots \\ 1 & P(z_{n-1}) & \cdots & \widehat{D^k P(z_{n-1})} & \cdots & D^{n-1} P(z_{n-1}) \end{pmatrix} + O(z_n^{-n}).\end{align*} Hence the Laurent series of $W^{(n; -k)}$ about $z_n=0$ begins $$W^{(n; k)} = \frac{(-1)^n (n+1)!}{(2\pi i)^{n+2}} W^{(n-1; -k)} z_n^{-(n+1)} + O(z_n^{-n}).$$ By a similar computation we find $$W^{(n; n)} = \frac{(-1)^{n-1} n!}{(2\pi i)^{n+1}} W^{(n-1; n-1)} z_n^{-n} + O(z_n^{-n+1}).$$ Finally, since $$\Theta^{(n)}(\tau, z) = \prod_{i=1}^{n+1} \frac{2\pi i \vartheta(\tau, z_i)}{\vartheta'(\tau, 0)} = 2\pi i \Theta^{(n-1)} z_n + O(z_n^2),$$ we obtain \[ \frac{(-1)^n}{n!} \Theta^{(n)} \cdot \frac{W^{(n;k)}}{W^{(n;n)}} = \frac{(-1)^{n+1}}{(n+1)!} \Theta^{(n-1)} \cdot \frac{W^{(n-1; k)}}{W^{(n-1; n-1)}} + O(z_n)\] and therefore \[ \iota^* \Big(\frac{(-1)^n}{n!} \Theta^{(n)} \cdot \frac{W^{(n;k)}}{W^{(n;n)}} \Big) = \frac{(-1)^{n+1}}{(n+1)!} \Theta^{(n-1)} \cdot \frac{W^{(n-1; k)}}{W^{(n-1; n-1)}}. \qedhere \]

There seems to be no canonical set of Wirthm\"uller generators for the $D_n$ root system in the literature. For $n \le 8$, a system of weak Jacobi forms that are characterized by a special differential equation was proposed by Adler--Gritsenko \cite{AG2020}.

\subsection{Orthogonal modular forms}

Let $T$ be an even integral lattice of signature $(2, n)$ for some $n \ge 1$ and $T_{\mathbb{C}} = T \otimes \mathbb{C}$. Let $$\mathcal{D}^{\pm} = \{[Z] \in \mathbb{P}(T_{\mathbb{C}}): \; Z \cdot Z = 0, \; Z \cdot \overline{Z} > 0\}$$ and fix either of two of its connected components, denoted by $\mathcal{D}$. Then $\mathcal{D}$ is a Hermitian symmetric domain of type IV. The group $$\mathrm{O}^+(T) = \{\gamma \in \mathrm{O}(T): \; \gamma \mathcal{D} = \mathcal{D}\}$$ acts properly discontinuously on $\mathcal{D}$ and the quotient $\mathrm{O}^+(T) \backslash \mathcal{D}$ is called the orthogonal modular variety attached to $T$. \\
Let $\mathcal{A} = \{Z \in T_{\mathbb{C}} \backslash \{0\}: \; [Z] \in \mathcal{D}\}$ be the affine cone over $\mathcal{D}$. For a finite-index subgroup $\Gamma \le \mathrm{O}^+(T)$, a modular form of weight $k \in \mathbb{N}_0$ and level $\Gamma$ is defined as a holomorphic function $$f : \mathcal{A} \longrightarrow \mathbb{C}$$ satisfying $$f(\gamma Z) = f(Z), \; \gamma \in \Gamma; \quad f(t Z) = t^{-k} f(Z), \; t \in \mathbb{C}^{\times},$$ as well as a holomorphy condition at cusps of codimension one if $n \in \{1, 2\}$ (which is never the case in this paper). The most important modular group $\Gamma$ will be the discriminant kernel $$\mathrm{\tilde O}(T) = \{\gamma \in \mathrm{O}^+(T): \; \gamma x \in x + T \; \text{for all} \; x \in T^\vee\},$$ where $T^\vee$ is the dual lattice $$T^\vee = \{y \in T \otimes \mathbb{Q}: \; y \cdot T \subseteq \mathbb{Z}\}.$$ The space of modular forms of weight $k$ is denoted $M_k(\Gamma)$. \\

Now suppose that $T = H^{\oplus 2} \oplus L(-1)$ for a positive-definite lattice $L$, and write vectors $x \in T$ as $(x_1, x_2, x_3, x_4, y)$ with $(x_1,x_2), (x_3, x_4) \in H_{\mathbb{C}}$ and $y \in L_{\mathbb{C}}$, such that $$x^2 = x_1 x_2 + x_3 x_4 - y^2.$$ Then it is well-known that modular forms $f$ of level $\mathrm{\tilde O}(T)$ are uniquely determined by their Fourier--Jacobi expansions: $$f(z^2 - \tau w, 1, \tau, w, z) = \sum_{n=0}^{\infty} \phi_n(\tau, z) e^{2\pi i n w},$$ in which each $\phi_n$ is a Jacobi form of index $L(n)$. \\

For a holomorphic Jacobi form $$\phi = \sum_{n=0}^{\infty} \sum_{\ell \in L^{\vee}} c(n, \ell) q^n \zeta^{\ell} \in J_{k, L},$$ the \emph{Gritsenko lift} or additive Borcherds lift is a modular form $\mathrm{Lift}(\phi) \in M_k(\mathrm{\tilde O}(T))$, where $T = H \oplus H \oplus L(-1)$, with Fourier--Jacobi expansion $$\mathrm{Lift}(\phi) = \sum_{m=0}^{\infty} \Big( \phi \Big| V_m \Big)(\tau, z) e^{2\pi i m w}.$$
For $m \ge 1$, $V_m$ is the index-raising Hecke operator $$\phi|V_m(\tau, z) = \frac{1}{m} \sum_{\substack{ad=m \\ a,d>0}} \sum_{b\in \mathbb{Z}/d} a^k \phi\left( \frac{a\tau + b}{d}, az \right) = \sum_{n=-\infty}^{\infty} \sum_{\ell \in L^{\vee}} \Big( \sum_{d | (n,\ell,m)} d^{k-1} c\left( \frac{mn}{d^2}, \frac{\ell}{d} \right) \Big) q^n \zeta^{\ell},$$ while if $k \ge 4$ is even then $V_0$ is formally defined by $$\phi \Big|V_0(\tau, z) = -c(0,0) \cdot \frac{B_k}{2k}E_k(\tau),$$ and $\phi|V_0 = 0$ otherwise.
In particular, the first Fourier--Jacobi coefficient of $\mathrm{Lift}(\phi)$ is $\phi$ itself. \\

The (multiplicative) Borcherds lift is an exponential analogue of the additive lift that maps weakly holomorphic Jacobi forms of weight zero with integral Fourier coefficients to orthogonal modular forms whose zeros and poles lie on Heegner divisors (or Noether--Lefschetz divisors in the context of moduli of K3 surfaces). For a vector $\ell \in T^\vee$ of negative norm, the hyperplane $\ell^{\perp} = \{[z] \in \mathcal{D}: \; z \cdot \ell = 0\}$ is a Hermitian symmetric domain of type $\mathrm{O}(2, n-1)$. For a given class $\gamma \in T^\vee/T$ and $n \in \mathbb{Z} + \gamma^2 / 2$, $n < 0$, the union $$H_{n, \gamma} = \bigcup_{\substack{\ell \in \gamma \\ \ell^2 = n}} \ell^{\perp}$$ is locally finite and $\mathrm{\tilde O}(T)$-invariant and therefore descends to a holomorphic divisor on $\mathrm{\tilde O}(T) \backslash \mathcal{D}$, called a Heegner divisor of discriminant $n$. 

The construction of Borcherds \cite{Borcherds1998} in the language of Jacobi forms following Gritsenko--Nikulin \cite{GN1998} is the following:

\begin{theorem} Let $\phi = \sum_{n \ge -N}^{\infty} \sum_{\ell \in L^{\vee}} c(n, \ell) q^n \zeta^{\ell}$ be a weakly holomorphic Jacobi form of weight zero and index $L$ with integral Fourier coefficients. Then there exists a meromorphic modular form $\Psi$ for $\mathrm{\tilde O}(T)$, $T = H^{\oplus 2} \oplus L(-1)$, with the following properties:
\begin{enumerate}
\item $\Psi$ has weight $\frac{c(0,0)}{2}$ and a character of finite order.
\item The zeros and poles of $\Psi$ lie on Heegner divisors. If $r$ is a vector of the form $(0,0)\oplus (n,1) \oplus \ell$ with respect to the splitting $H \oplus H \oplus L(-1)$, then the order of $\Psi$ along $r^{\perp}$ is $$\mathrm{ord}(\Psi; r^{\perp}) = \sum_{m=1}^{\infty} c(m^2 n, m \ell).$$
\item The Fourier--Jacobi expansion of $\Psi$ is the product $$\Psi = e^{2\pi i C w} \Theta(\tau, z) \times \exp \Big( \sum_{n=1}^{\infty} (\phi | V_n)(\tau, z) e^{2\pi i n w} \Big),$$ where $\Theta$ is the theta block \begin{align*} \Theta(\tau, z) &:= \eta(\tau)^{c(0,0)} \prod_{\ell > 0} \Big( \frac{\vartheta(\tau, \ell \cdot z)}{\eta(\tau)} \Big)^{c(0, \ell)} \\ &= q^{\frac{1}{24} \sum_{\ell \in L^{\vee}} c(0, \ell)} \prod_{\substack{\ell \in L^{\vee} \\ \ell > 0}} (\zeta^{\ell/2} - \zeta^{-\ell/2}) \prod_{\ell \in L^{\vee}} \prod_{n=1}^{\infty} (1 - q^n \zeta^{\ell})^{c(0, \ell)},\end{align*} and $C$ is the constant $$C = \frac{1}{2 \cdot \mathrm{rank}(L)} \sum_{\ell \in L^{\vee}} \ell^2 c(0, \ell).$$
\end{enumerate}
\end{theorem}

The graded ring $M_*(\Gamma) = \bigoplus_{k=0}^{\infty} M_k(\Gamma)$ of modular forms of all weights is known to be finitely-generated by the theorem of Baily--Borel. For certain modular groups that are related to root lattices, the ring structure is especially simple:

\begin{theorem}\label{thrm:freealg} \cite{WW20} Let $R$ be one of the root lattices $A_n$ $(1 \le n \le 7)$, $B_n$ $(2 \le n \le 4)$, $C_n$ $(3 \le n \le 8)$, $D_n$ $(3 \le n \le 8)$, $E_6, E_7, F_4$ or $G_2$. Let $L$ be the root lattice and let $T = H^{\oplus 2} \oplus L(-1)$, and let $\Gamma = \Gamma_R$ be the subgroup of $\mathrm{O}^+(T)$ generated by the discriminant kernel $\mathrm{\tilde O}(T)$ and by the Weyl group $W_R$ acting on $L$. Then: \\ (i) $M_*(\Gamma)$ is a free algebra. \\ (ii) If $\phi_{-k, m} \in J_{*, L(*)}^{\mathrm{w}, W}$ are generators of the Wirthm\"uller ring of weak Jacobi forms, then the generators of $M_*(\Gamma)$ can be chosen to have leading Fourier--Jacobi coefficients $\Delta^m \phi_{-k,m}.$ Conversely, any system of modular forms whose leading Fourier--Jacobi coefficients are $\Delta^m \phi_{-k,m}$ is a set of generators for $M_*(\Gamma)$.
\end{theorem}
In particular, $M_*(\Gamma)$ is generated in weights $4, 6$ and $12m-k$, where $(k,m)$ are the numbers in Theorem \ref{theorem:wirthmueller}.

%Free algebras of modular forms attached to root systems were studied in \cite{WW20}. Here, the root systems appear in the lattice $T$, or the transcendental lattice of the general K3 surface in the family, rather than the N\'eron-Severi lattice. For an irreducible root system $R$, let $L$ be its root lattice as given by the following table: 

The lattices $L$ associated to the irreducible root systems are the following ADE lattices:

\begin{table}[htpb]
\begin{tabular}{|l|l|l|l|l|l|l|l|}
\hline
$R$   & $A_m$ & $B_m$            & $C_m$ & $D_m$ & $E_m$ & $F_4$ & $G_2$ \\ \hline
&&&&&&& \\[-1em]
$L$ & $A_m$ & $A_1^{\oplus m}$ & $D_m$ & $D_m$ & $E_m$ & $D_4$ & $A_2$ \\ \hline
\end{tabular}
\end{table}

The exceptional root system $E_8$ does not fit into this pattern as its ring of weak Jacobi forms is known not to be freely generated \cite{Wang21}. Nevertheless, the algebra $M_*(\Gamma_{E_8})$ is free, as shown in \cite{HU22}. Conversely, these account for all free algebras of modular forms for lattices splitting two hyperbolic planes, as shown in \cite{Wang212}.

%The Weyl group $W(R)$ naturally acts on $L$. Let $T = H \oplus H \oplus L(-1)$ and let $\Gamma_R$ be the subgroup of $\mathrm{O}^+(T)$ generated by the discriminant kernel $\mathrm{\tilde O}(T)$ and by $W(R)$. (See section 3.3 below for the definitions.) For this group, the structure of the ring of modular forms is known. Let us recall following theorem (cf. \cite{WW20}, except for the exceptional case $R=E_8$ which is due to \cite{HU22}):

%\begin{theorem}[\cites{WW20,HU22}] \label{thrm:freealg}
%    For $R \in \{A_1,...,A_7,B_1,...,B_4,C_1,...,C_8,D_1,...,D_8,E_6,E_7,E_8,F_4,G_2\}$, the graded ring of modular forms $M_*(\Gamma_R)$ is freely generated. The weights of the generating modular forms are uniquely determined by $R$ and are given in Figure \ref{fig:freealg}.
\begin{figure}[htpb]
\begin{tabular}{|l|l|}
\hline
$R$   & Weights       \\ \hline \hline
$A_1$ & $4,6,10,12$   \\ \hline
$A_2$ & $4,6,9,10,12$ \\ \hline
$B_2$ & $4,6,8,10,12$ \\
\hline
$G_2$ & $4,6,10,12,18$ \\
\hline
$A_3$ & $4,6,8,9,10,12$ \\
\hline
$B_3$ & $4, 6, 6, 8, 10, 12$ \\
\hline
$C_3$ & $4, 6, 8, 10, 12, 18$ \\
\hline
$A_4$ & $4, 6, 7, 8, 9, 10, 12$ \\
\hline
$B_4$ & $4, 4, 6, 6, 8, 10, 12$ \\
\hline
$C_4$ & $4, 6, 8, 10, 12, 16, 18$ \\
\hline
$D_4$ & $4, 6, 8, 8, 10, 12, 18$ \\
\hline
$F_4$ & $4, 6, 10, 12, 16, 18, 24$ \\
\hline
$A_5$ & $4, 6, 6, 7, 8, 9, 10, 12$ \\
\hline
$C_5$ & $4, 6, 8, 10, 12, 14, 16, 18$ \\
\hline
$D_5$ & $4, 6, 7, 8, 10, 12, 16, 18$ \\
\hline
$A_6$ & $4, 5, 6, 6, 7, 8, 9, 10, 12$ \\
\hline
$C_6$ & $4, 6, 8, 10, 12, 12, 14, 16, 18$ \\
\hline
$D_6$ & $4, 6, 6, 8, 10, 12, 14, 16, 18$ \\
\hline
$E_6$ & $4, 6, 7, 10, 12, 15, 16, 18, 24$ \\
\hline
$A_7$ & $4, 4, 5, 6, 6, 7, 8, 9, 10, 12$ \\
\hline
$C_7$ & $4, 6, 8, 10, 10, 12, 12, 14, 16, 18$ \\
\hline
$D_7$ & $4, 5, 6, 8, 10, 12, 12, 14, 16, 18$ \\
\hline
$E_7$ & $4, 6, 10, 12, 14, 16, 18, 22, 24, 30$ \\
\hline
$C_8$ & $4, 6, 8, 8, 10, 10, 12, 12, 14, 16, 18$ \\
\hline
$D_8$ & $4, 4, 6, 8, 10, 10, 12, 12, 14, 16, 18$ \\
\hline
$E_8$ & $4, 10, 12, 16, 18, 22, 24, 28, 30, 36, 42$ \\
\hline
\end{tabular}
\caption{\label{fig:freealg} Irreducible root systems and free algebras of modular forms.}
\end{figure}

Many cases of Theorem \ref{thrm:freealg} were already known before the work of \cite{WW20}. In particular, the algebra structure in the case $R = D_n$ is due to Vinberg \cite{Vinberg2010}, \cite{Vinberg2018}.

\begin{corollary}\label{cor:An} Let $T = H^{\oplus 2} \oplus A_n(-1)$, $1 \le n \le 7$. Then $M_*(\mathrm{\tilde O}(T)) = \mathbb{C}[F_6, F_4, \chi_k]$, $k = 12, 10, 9, ..., 11-n$ with generators $$\chi_k = \mathrm{Lift}( \Delta \phi_{k-12}^{(n)})$$ and the Eisenstein-type generators defined as follows: if $n=7$, $$F_4 := 240\cdot \mathrm{Lift} \Big( (1 + 2q + 2q^4 + ...) (\vartheta_0 + \vartheta_{4\ell}) + (2q^{1/4} + 2q^{9/4} + ...) (\vartheta_{2\ell} + \vartheta_{4\ell}) \Big);$$ $$F_6 := -504\cdot \mathrm{Lift} \Big( (1 - 70q - 120q^2 - ...) (\vartheta_0 + \vartheta_{4\ell}) + (-10q^{1/4} - 48q^{5/4} - ...) (\vartheta_{2\ell} + \vartheta_{6\ell}) \Big),$$ and for $n < 7$, they are defined by the pullback $F_4 = \iota^* F_4$, $F_6 = \iota^* F_6$ with respect to the embedding $\iota : A_n \rightarrow A_7$.
\end{corollary}

In the definition of $F_4$ and $F_6$, $\vartheta_{m\ell}$ denotes the theta function $\vartheta_{m\ell + A_7}$ where $\ell = \frac{1}{n+1}(1,...,1, -n) = \frac{1}{8}(1,...,1,-7)$ is the fixed generator of $A_7^{\vee}/A_7$. The factors $240$ and $-504$ guarantee that the result of the Gritsenko lift has integral Fourier coefficients.
\end{proof}
\section{The \texorpdfstring{$A_n$}{An}-family of K3 surfaces}

Recall that $\ell$ is the fixed generator of the discriminant group $A_n^{\vee} / A_n$, represented by $\frac{1}{n+1} (1,...,1,-n)$ with respect to the convention that $A_n$ is the lattice of integral $(n+1)$-vectors that sum to zero.

\begin{theorem} \label{thm:weq_An}
For $1 \le n \le 7$, define families of Weierstrass equations as follows:

\begin{table}[H]
\caption{The $A_n$-family of K3 surfaces}
\begin{tabular}{|l|l|}
\hline
$n=1$ & $y^2 = x^3 + (u^3 + f_4 u + f_6) x^2 + (g_{10} u + g_{12}) x$ \\ \hline
$n=2$ & $y^2 = x^3 + (u^3 + f_4 u + f_6) x^2 + (g_{10} u + g_{12}) x + h_9^2$ \\ \hline
$n=3$ & $y^2 = x^3 + (u^3 + f_4 u + f_6) x^2 + (g_8 u^2 + g_{10} u + g_{12}) x + h_9^2$ \\ \hline
$n=4$ & $y^2 = x^3 + (u^3 + f_4 u + f_6) x^2 + (g_8 u^2 + g_{10} u + g_{12}) x + (h_7 u + h_9)^2$  \\ \hline
$n=5$ & $y^2 = x^3 + (u^3 + f_4 u + f_6) x^2 + (g_6 u^3 + g_8 u^2 + g_{10} u + g_{12}) x + (h_7 u + h_9)^2$ \\ \hline
$n=6$ & $y^2 = x^3 + (u^3 + f_4 u + f_6) x^2 + (g_6 u^3 + g_8 u^2 + g_{10} u + g_{12}) x + (h_5u^2 + h_7 u + h_9)^2$  \\ \hline
$n=7$ & $y^2 = x^3 + (u^3 + f_4 u + f_6) x^2 + (g_4 u^4 + g_6 u^3 + g_8 u^2 + g_{10} u + g_{12}) x + (h_5 u^2 + h_7 u + h_9)^2$  \\ \hline
\end{tabular}
\label{tab:An}
\end{table}
\par For a very general choice of parameters, the equation above describes a Jacobian K3 surface $\mathcal{X}$ of multiplicity one whose transcendental lattice is $T = H^{\oplus  2} \oplus A_n(-1)$. The coefficients $f_k, g_k, h_k$ define algebraically independent modular forms for the discriminant kernel $\mathrm{\tilde O}(T)$ of the weight indicated by their subscript, and they freely generate the graded algebra $M_*(\mathrm{\tilde O}(T))$.

\end{theorem}
\begin{proof}
For $n \ge 2$, the elliptic fibrations of Table \ref{tab:An} generically have a singular fiber at $u=\infty$ of type $I_{11-n}^*$ and $(n+7)$ singular fibers of type $I_1$. The Weierstrass models are obtained by moving the reducible singular fiber to $u=\infty$, then normalizing the coefficient of $u^3x^2$ to one, and eliminating the subleading term $u^2x^2$. Overall shifts in $x$ are then constrained by requiring that the constant term is a perfect square. Rescaling $(u, x, y) \mapsto (\lambda^2 u, \lambda^6 x, \lambda^9y)$ for $\lambda \in \mathbb{C}^\times$ constitutes a holomorphic isomorphism, and the coefficients have the weights indicated by their subscript. 

Accordingly, for $n \ge 2$ the elliptic fibrations yield the trivial lattice $H \oplus D_{15-n}(-1)$. The N\'eron--Severi lattice is the extension of this lattice by the infinite-order section $\sigma'$ with $x=0$. The section $\sigma'$ does not intersect the neutral component of the reducible fiber of type $D_{15-n}$. Moreover, $\sigma'$ does not intersect the zero section $\sigma$  whence $\sigma \circ \sigma'=0$.  Thus, its height pairing is
\begin{equation}
\begin{aligned}
  \langle \sigma', \sigma' \rangle & = 4 - C^{-1}_{D_{15-n}}(\sigma', \sigma'),
\end{aligned}
\end{equation}
where $C^{-1}_{D_{15-n}}$ denotes the inverse Cartan matrix. A section can only intersect the reducible fiber in a terminal node, connected to one of the two forks in the extended Dynkin diagram $\widetilde{D}_{15-n}$. In particular, $\sigma'$ intersects a node connected with the fork opposite to one containing the neutral component.  The reason is that the $y$-coordinate of the section $\sigma'$ vanishes at $u=\infty$ with degree $4$, i.e.~it runs through the singularity of the reduced Weierstrass model and still does so after undoing a quadratic twist. Thus, the height pairing is
\begin{equation}
\begin{aligned}
      \langle \sigma', \sigma' \rangle & = 4 - \frac{15-n}{4} = \frac{n+1}{4}.
\end{aligned}
\end{equation}
The polarizing lattice is generated by
\[
\mathrm{NS}_n := \mathrm{NS}(\mathcal{X}) \ \cong \Big\langle \sigma + f, f , \ d_1, d_2, \dots, d_{15-n}, \ \sigma' - \sigma - 2f \Big \rangle,
\]
where $d_1, \dots, d_{15-n}$ are the nodes of the $D_{15-n}$-graph, with $d_1 \circ \sigma'=1$ and $d_i \circ \sigma'=0$ for $i=2, \dots, 15-n$. The lattice has rank $18-n$ and discriminant $n+1$. For the choice of generators the lattice decomposes in the form $H \oplus K_n(-1)$ where $K_n$ is a positive definite lattice of rank $16-n$.

For each integer $n = 1, \dots, 7$, we have constructed the N\'eron--Severi lattice $\mathrm{NS}_n$ from the Gram matrices. Let $T_n = H \oplus H \oplus A_n(-1)$ be the candidate transcendental lattice. We then verified in each case that $\mathrm{NS}_n$ and $T_n$ satisfy the conditions of Nikulin's theorem on primitive orthogonal complements in the K3 lattice $\Lambda_{\mathrm{K3}}$.  Explicitly, the sum of their ranks equals $22$, and the sum of their signatures equals $(3,19)$, matching the signature of $\Lambda_{\mathrm{K3}}$. Moreover, their discriminant groups are easily shown to be isomorphic, with identical invariants.

It remains to compare the discriminant quadratic forms. In each case, the discriminant group is cyclic of order $d = n+1$, so the quadratic form is determined by a value $q(x) = \frac{a}{d} \in \mathbb{Q}/2\mathbb{Z}$ on a generator. Two such forms are isomorphic if and only if they differ by multiplication by a square in $(\mathbb{Z}/d\mathbb{Z})^\times$. Thus, Nikulin’s condition requires that there exists a unit $k \in (\mathbb{Z}/d\mathbb{Z})^\times$ such that
\[
q_{\, \mathrm{NS}_n}(x) + q_{\, T_n}(k x) \equiv 0 \pmod{2\mathbb{Z}},
\]
or equivalently,
\[
\frac{a}{d} + \frac{b k^2}{d} \equiv 0 \pmod{2}.
\]
A direct check shows that such a unit exists for each $n = 1, \dots, 7$. For $n = 1,2,3,5$, one may take $k=1$, while for $n=4,6$ and $n=7$ one may take $k=2$ and $k=3$, respectively. Hence the discriminant forms satisfy $q_{\, \mathrm{NS}_n} \cong - q_{\, T_n}$, and all conditions of Nikulin’s theorem are fulfilled.

By construction, the N\'eron--Severi lattice decomposes as $\mathrm{NS}_n \cong H \oplus K_n(-1)$. Since $H$ is unimodular, the discriminant form of $\mathrm{NS}_n$ coincides with that of $K_n$. We consider the index
\[
\big[ \mathrm{O}(q_{K_n}) : \operatorname{Im}{\mathrm{O}(K_n)} \big].
\]
The quantity measures the multiplicity of elliptic fibrations with fixed frame. A direct computation shows that the multiplicity equals one for $n = 1, \dots, 7$.

Finally, the fact that the multiplicity equals one implies that, for very general parameters, the associated Jacobian elliptic fibration is uniquely determined by its frame, up to automorphisms of the surface. In particular, there are no additional identifications arising from nontrivial automorphisms of the discriminant form. It follows that the parameter space of the Weierstrass models coincides with the moduli space associated with the discriminant kernel $\mathrm{\tilde O}(T_n)$. Consequently, the coefficients $f_k$, $g_k$, and $h_k$ define modular forms for $\mathrm{\tilde O}(T_n)$ of the indicated weights, and they freely generate the graded algebra $M_*(\mathrm{\tilde O}(T_n))$.
\end{proof}

\begin{proof}[\textbf{Proof of Theorem~\ref{thm:main}\,\ref{thm:partII}}]
Theorem~\ref{thm:weq_An} has established the Weierstrass normal form for the general $S$-polarized K3 surface with $T = S^{\perp} = H \oplus H \oplus A_n(-1)$. We now construct the coefficients $f_k, g_k, h_k$ as modular forms with respect to the orthogonal group $\mathrm{\tilde O}(T)$ explicitly. To do so we start with the case $n=1$, and then construct iteratively the cases for higher $n$ until $n=7$ is reached.
\subsection{\texorpdfstring{$n=1$}{n equals 1}} 
The discriminant kernel $\mathrm{\tilde O}(T)$ for $T=H^{\oplus 2} \oplus A_1(-1)$ is closely related to $\mathrm{Sp}_4(\mathbb{Z})$ via the exceptional isogeny from $\mathrm{Sp}_4$ to $\mathrm{SO}(2,3)$. This allows modular forms for $\mathrm{\tilde O}(T)$ to be identified with the ring of Siegel modular forms of degree two and even weight.

The full ring of Siegel modular forms of degree two was calculated by Igusa \cite{Igusa62}; it decomposes as $$\mathbb{C}[E_4, E_6, \chi_{10}, \chi_{12}] \oplus \chi_{35} \mathbb{C}[E_4, E_6, \chi_{10}, \chi_{12}]$$ where $\chi_{35}$ satisfies a quadratic relation of the form $$\chi_{35}^2 = \chi_{10} \cdot (\text{irreducible expression of weight} \; 60).$$ The functions $E_4$ and $E_6$ are Siegel Eisenstein series, coinciding with the forms $F_4$ and $F_6$ introduced in \ref{cor:An}, and the forms $\chi_{10}$ and $\chi_{12}$ equal the standard Igusa cusp forms up to a normalization factor that guarantees that their Fourier coefficients are coprime integers. \\

Taking the discriminant with respect to $x$ in the defining Weierstrass equation $$y^2 = x^3 + (u^3 + f_4 u + f_6) x^2 + (g_{10} u + g_{12}) x$$ yields $$\mathrm{Disc}_x \left( x^3 + (u^3 + f_4 u + f_6) x^2 + (g_{10} u + g_{12}) x \right) = (g_{10} u + g_{12})^2 \cdot P_6(u)$$ with an irreducible polynomial $P_6$, $\mathrm{deg}(P_6) = 6$ that describes the $I_1$ fibers at $u = \beta_i$, $i=1,...,6$. We have $$D := \mathrm{Disc}_u \big(P_6\big) = \Big( \prod_{i < j} (\beta_i - \beta_j) \Big)^2$$ and $\mathrm{deg}(D) = 60$. Hence $D$ is the square of a Siegel modular form of weight $30$ with a quadratic character. This is already sufficient to determine $$D = \mathrm{const} \cdot \frac{\chi_{35}^2}{\chi_{10}}.$$
Since the coefficients $f_4, f_6, g_{10}, g_{12}$ belong to spaces of Siegel modular forms of dimensions $1, 1, 2, 3$, comparing a few Fourier coefficients between $D$ and $\chi_{35}^2 / \chi_{10}$ is enough to determine the values $$f_4 = -3 E_4 = -3 F_4, \quad f_6 = 2 E_6 = 2 F_6$$ and $$g_{10} = -12^5 \cdot \chi_{10}, \quad g_{12} = 12^5 \cdot \chi_{12},$$ and therefore the universal Weierstrass surface $$y^2 = x^3 + (u^3 - 3 F_4 u + 2 F_6) x^2 + 12^5 (-\chi_{10} u + \chi_{12}) x.$$
\subsection{\texorpdfstring{$n=2$}{n equals 2}} 
The coefficients of the Weierstrass equation $$y^2 = x^3 + (u^3 + f_4 u + f_6) x^2 + (g_{10} u + g_{12}) x + h_9^2$$ are modular forms for the discriminant kernel of $T = H^{\oplus 2} \oplus A_2(-1)$. \\

The vanishing locus of the coefficient $h_9$ exactly cuts out the locus of K3 surfaces where the general member has transcendental lattice $T_1 = H^{\oplus 2} \oplus A_1(-1)$. In the period domain, this is the Heegner divisor of discriminant $-\frac{1}{2} \frac{|\mathrm{det}(T_1)|}{|\mathrm{det}(T)|} = -\frac{1}{3}$. In fact, $h_9$ is a Borcherds product lifted from the weak Jacobi form $\phi_0^{(2)}$. \\

By the Koecher principle, any modular form for $\mathrm{\tilde O}(T)$ whose specialization to $T_1$ vanishes must be a multiple of $h_9$. It follows that the coefficients $f_4, f_6, g_{10}, g_{12}$ are uniquely determined by their specializations: $$f_4 = -3 F_4, \; f_6 = 2 F_6, \; g_{10} = -12^5 \chi_{10}, \; g_{12} = 12^5 \chi_{12},$$ exactly as in the previous subsection, where $F_k, \chi_k$ now denote the lifts to $\mathrm{\tilde O}(T)$. \\

The remaining coefficient $h_9$ is necessarily a multiple of $\chi_9$. The multiple can be determined (up to a sign) by a similar trick as before, by defining $$D(u) = \mathrm{Disc}_x \Big( x^3 + (u^3 + f_4 u + f_6) x^2 + (g_{10} u + g_{12}) x + h_9^2 \Big)$$ and observing that $\mathrm{Disc}_u(D)$ is a homogeneous expression in $f_4,f_6,g_{10},g_{12},h_9$ of degree $90$ which equals, up to a multiple, the square of a certain reflective Borcherds product of weight $45$. Comparing Fourier coefficients then yields $h_9 = \pm 12^{9/2} \chi_9$. We fix $h_9 = 12^{9/2} \chi_9$ for simplicity and obtain the Weierstrass model $$y^2 = x^3 + (u^3 - 3 F_4 u + 2 F_6) x^2 + 12^5 (-\chi_{10} u + \chi_{12})x + 12^9 \chi_9^2.$$

\begin{remark} The modular forms for $\mathrm{\tilde O}(T)$ considered here were constructed in the context of Hermitian modular forms by Dern--Krieg \cite{DK03}.
\end{remark}
\subsection{\texorpdfstring{$n=3$}{n equals 3}} 
The coefficients of the Weierstrass equation $$y^2 = x^3 + (u^3 + f_4 u + f_6) x^2 + (g_8 u^2 + g_{10} u + g_{12}) x + h_9^2$$ are modular forms for the discriminant kernel of $T = H^{\oplus 2} \oplus A_3(-1)$. \\

The observation of the previous subsection applies to $g_8$ and shows that it is a Borcherds product whose zeros lie exactly on the Heegner divisors of discriminant $-\frac{3}{8}$, where the transcendental lattice of the general K3 surface becomes $H^{\oplus 2} \oplus A_2(-1)$. By another application of Koecher's principle, we immediately find the coefficients $f_4, f_6, g_{10}, h_9$: $$f_4 = -3F_4, \; f_6 = 2 F_6, \; g_{10} = -12^5 \chi_{10}, \; h_9 = 12^{9/2} \chi_9.$$

The form $\chi_9$ turns out also to be a Borcherds product on $\mathrm{\tilde O}(T)$, vanishing exactly on Heegner divisors of discriminant $-1/2$, and the specializations of the remaining coefficients to the locus $\chi_9 = 0$ define modular forms for $\mathrm{\tilde O}(H^{\oplus 2} \oplus A_1(-1)^{\oplus 2})$. Indeed, setting $h_9 = 0$ in the Weierstrass equation yields a K3 surface with two-elementary N\'eron--Severi lattice carrying the so-called alternate fibration $$y^2 = x^3 + (u^3 + f_4 u + f_6) x^2 + (g_8 u^2 + g_{10} u + g_{12})x,$$ whose coefficients were studied in \cite{MR4987329}. Compatibility with the results of \cite{MR4987329} forces the remaining coefficients to be $$g_8 = -\frac{12^5}{2} \chi_8, \quad g_{12} = 12^5 \chi_{12} + \frac{12^5}{2} \chi_8 F_4.$$ Altogether, we have the model $$y^2 = x^3 + (u^3 - 3 F_4 u + 2 F_6) x^2 + 12^5 \Big( -\frac{1}{2} \chi_8 u^2 - \chi_{10} u + \chi_{12} + \frac{1}{2} \chi_8 F_4 \Big) x + 12^9 \chi_9^2.$$
\subsection{\texorpdfstring{$n=4$}{n equals 4}}
The Weierstrass equation is now $$y^2 = x^3 + (u^3 + f_4 u + f_6) x^2 + (g_8 u^2 + g_{10} u + g_{12}) x + (h_7 u + h_9)^2.$$

Similarly to the earlier subsections, $h_7$ is a Borcherds product with zeros exactly on the Heegner divisors of discriminant $-\frac{2}{5}$, and its vanishing locus exactly describes the $n=3$ family of K3 surfaces. An appeal of Koecher's principle shows that \emph{all} coefficients other than $h_7$ are uniquely determined by their specializations: $$f_4 = -3F_4, \; f_6 = 2 F_6, \; g_8 = -\frac{12^5}{2} \chi_8, \; g_{10} = -12^5 \chi_{10}$$ $$g_{12} = 12^5 \chi_{12} + \frac{12^5}{2} \chi_8 F_4, \; h_9 = 12^{9/2} \chi_9.$$

To determine $h_7$, we specialize to the case that the coefficients $g_8 u^2 + g_{10} u + g_{12}$ and $h_7 u + h_9$ share a common root $\beta$, such that the equation becomes $$y^2 = x^3 + (u^3 + f_4 u + f_6) x^2 + (u - \beta) (g_8 u + \alpha_{10}) + (u - \beta)^2 h_7^2$$ with a certain coefficient $\alpha_{10}$. This occurs exactly on the vanishing locus of the resultant $$R_{26} = \mathrm{Res}_u \Big( g_8 u^2 + g_{10} u + g_{12}, h_7 u + h_9 \Big) = g_8 h_9^2 - g_{10} h_7 h_9 + g_{12} h_7^2,$$ which is therefore a certain Borcherds product of weight $26$. By comparing leading Fourier--Jacobi coefficients, we obtain $h_7 = 12^{9/2} \chi_7$ and therefore the Weierstrass model $$y^2 = x^3 + (u^3 - 3 F_4 u + 2 F_6) x^2 + 12^5 \Big( -\frac{1}{2} \chi_8 u^2 - \chi_{10} u + \chi_{12} + \frac{1}{2} \chi_8 F_4 \Big) x + 12^9 (\chi_7 u + \chi_9)^2.$$

\begin{remark}
The coefficient $g_{12}$ can be expressed as a single Gritsenko lift: $$g_{12} = 12^5 \cdot \mathrm{Lift}\Big( \Delta \phi_0^{(4)} + \frac{1}{2}\Delta E_4 \phi_{-4}^{(4)} \Big).$$
\end{remark}
\subsection{\texorpdfstring{$n=5$}{n equals 5}}
The Weierstrass equation is now \begin{equation} \label{eq: A5} y^2 = x^3 + (u^3 + f_4 u + f_6) x^2 + (g_6 u^3 + g_8 u^2 + g_{10} u + g_{12}) x + (h_7 u + h_9)^2\end{equation} and the coefficients $f_k, g_k, h_k$ are modular forms of weight $k$ for $\mathrm{\tilde O}(T)$, $T = H^{\oplus 2} \oplus A_5(-1)$. \\

$g_6$ is a Borcherds product with zeros exactly on the Heegner divisors of discriminant $-5/12$, and its vanishing locus cuts out the family of K3 surfaces of the previous subsection. Koecher's principle yields the expressions $$f_4 = -3 F_4, \; g_8 = -\frac{12^5}{2} \chi_8, \; h_7 = 12^{9/2} \chi_7, \; h_9 = 12^{9/2} \chi_9.$$

To determine the remaining coefficients, observe that at $x=-g_6$ the Weierstrass polynomial simplifies to become a quadratic polynomial in $u$: $$y^2 = (h_7^2 - g_6 g_8) u^2 + (f_4 g_6^2 - g_6 g_{10} + 2h_7 h_9) u + f_6 g_6^2 - g_6^3 - g_6 g_{12} + h_9^2.$$ The vanishing of the discriminant of that polynomial, \begin{align*} D &=f_{4}^{2} g_{6}^{3} + 4 f_{6} g_{6}^{2} g_{8} - 4 g_{6}^{3} g_{8} - 2 f_{4} g_{6}^{2} g_{10} - 4 f_{6} g_{6} h_{7}^{2} \\ &\quad + 4 g_{6}^{2} h_{7}^{2} + 4 f_{4} g_{6} h_{7} h_{9} + g_{6} g_{10}^{2} - 4 g_{6} g_{8} g_{12} \\ &\quad + 4 g_{12} h_{7}^{2} - 4 g_{10} h_{7} h_{9} + 4 g_{8} h_{9}^{2},\end{align*} is equivalent to the existence of a second infinite-order section of the form $x(u) = -g_6$, $y(u) = au+b$. Hence $D$ is a modular form that vanishes exactly on certain Noether--Lefschetz divisors; indeed, $D$ is a Borcherds product of weight $26$ with zeros exactly on Heegner divisors of discriminant $-3/4$. Comparing the Fourier series of $D$ with the product expansion of that Borcherds product uniquely determines all other coefficients:

 \begin{align*} f_6 &= 2F_6 + 414720 \chi_6 \\ &= 2 \cdot \mathrm{Lift} \Big( -504 e_0 + 720 q^{1/4} e_{3\ell} + 2808 q^{1/3} (e_{2\ell} + e_{-2\ell}) + 10800 q^{7/12} (e_{\ell} + e_{-\ell}) + 40896q e_0 + ... \Big); \\ g_6 &= -207360 \chi_6 = -1728 \cdot \mathrm{Lift} \Big( q^{1/4} e_{3\ell} - q^{1/3} (e_{2\ell} + e_{-2\ell}) + q^{7/12} (e_{\ell} + e_{-\ell}) - 2q e_0 \pm ... \Big); \\ g_{10} &= -12^5 \chi_{10} + 2^9 3^6 F_4 \chi_6 = -5184 \cdot \mathrm{Lift} \Big( q^{1/4} e_{3\ell} + 3q^{1/3} (e_{2\ell} + e_{-2\ell}) + 9 q^{7/12} (e_{\ell} + e_{-\ell}) - 218 q e_0 \pm ... \Big); \\ g_{12} &= 12^5 \chi_{12} + \frac{12^5}{2} \chi_8 F_4 - 2^{11} \cdot 3^4 \chi_6 F_6 - 2^{18} \cdot 3^{11} \cdot 5 \chi_6^2 \\ &= 1728 \Phi_{12} - \frac{3}{2} g_6^2, \quad \Phi_{12} = \mathrm{Lift} \Big( q^{1/4} e_{3\ell} + 5 q^{1/3} (e_{2\ell} + e_{-2\ell}) + 49q^{7/12} (e_{\ell} + e_{-\ell}) + 970q e_0 \pm ... \Big). \end{align*}

Note that this is the first model where one of the coefficients ($g_{12}$) does not have a representation as a Gritsenko lift. 

Altogether, we have the model \begin{align*} y^2 &= x^3 + (u^3 - 3 F_4 u + 2F_6 + 12^4 \cdot 20 \chi_6)x^2 \\ &+ 12^5 \Big( -\frac{5}{6} \chi_6 u^3 - \frac{1}{2} \chi_8 u^2 + (- \chi_{10} + \frac{3}{2} F_4 \chi_6) u + \chi_{12} + \frac{1}{2} F_4 \chi_8 - \frac{2}{3} F_6 \chi_6 - 2^8 \cdot 3^6 \cdot 5 \chi_6^2 \Big) x \\ &+ 12^9 (\chi_7 u + \chi_9)^2. \end{align*}

\begin{remark} The polynomial $D$ also arises naturally from the Jacobian elliptic fibration given by the same equation but viewing $x$ as the base variable on $\mathbb{P}^1$ and $u$ as the fiber variable: $$y^2 = (x^2 + g_6 x) u^3 + (g_8 x + h_7^2) u^2 + (f_4 x^2 + g_{10} x + 2 h_7 h_9) u + (x^3 + f_6 x^2 + g_{12} x + h_9^2),$$ which generically has singular fibers of Kodaira symbol $II^*$ at $x=\infty$, symbol $I_3$ at $x=0$, symbol $I_2$ at $x=-g_6$ and symbol $I_1$ at nine other points. The singular fibers of this fibration give another description of the N\'eron--Severi lattice as $$\mathrm{NS}(\mathcal{X}) = H \oplus E_8(-1) \oplus A_2(-1) \oplus A_1(-1).$$ The vanishing of the discriminant $D$ determines the collision of an $I_1$--fiber with the $I_2$--fiber at $x=-g_6$.
\end{remark}

We can also consider the special case of Equation \ref{eq: A5} where we demand the coefficients of $x^1$ and $x^0$ to have a common root: $$y^2 = x^3 + (u^3 + f_4 u + f_6) x^2 + (u - \beta_2) (g_6 u^2 + \alpha_8 u + \alpha_{10}) + (u - \beta_2)^2 h_7^2.$$
This imposes a type $I_2$ singular fiber at $u=\beta_2$ which changes the general transcendental lattice to $T_1 = H^{\oplus 2} \oplus A_1(-1) \oplus A_3(-1)$. In terms of the Weierstrass coefficients, this is the vanishing locus of the resultant $$R_{33} = g_6 h_9^3 - g_8 h_7 h_9^2 + g_{10} h_7^2 h_9 - g_{12} h_7^3,$$ which is therefore a Borcherds product of weight $33$ with zeros exactly on Heegner divisors of discriminant $-2/3$.
The coefficients $\beta_2, f_4, f_6, g_6, h_7, \alpha_8, \alpha_{10}$ generate a free algebra of meromorphic modular forms of a class considered in \cite{WW21} associated to reducible root system (in this case, the root system $A_1 + A_3$); in the terminology of \cite{WW21}, $\beta_2$ is of Abelian type, $f_4, f_6$ of Eisenstein type, and $g_6, h_7, \alpha_8, \alpha_{10}$ of Jacobi type. Explicit expressions for these coefficients in terms of Gritsenko lifts are straightforward to describe by means of the pullback trick. For later purposes we mention that $\beta_2$ is the singular Gritsenko lift
$$\beta_2 = 12 \cdot \mathrm{Lift}(\phi_{-2}^{(1)} \otimes E_4),$$ where $E_4$ is the Jacobi Eisenstein series of $A_3$ lattice index, and $\otimes$ is the ``exterior product" of Jacobi forms $$f \otimes g(\tau, z_1,...,z_{m+n}) = f(\tau, z_1,...,z_m) \cdot g(\tau, z_{m+1},...,z_{m+n}).$$
This follows from the identity $$\iota^* h_7 \cdot \beta_2 + \iota^* h_9 = 0$$ for the inclusion $\iota : A_1 \oplus A_3 \rightarrow A_5$.
\subsection{\texorpdfstring{$n=6$}{n equals 6}}
The Weierstrass equation is now $$y^2 = x^3 + (u^3 + f_4 u + f_6) x^2 + (g_6 u^3 + g_8 u^2 + g_{10} u + g_{12}) x + (h_5 u^2 + h_7 u + h_9)^2$$ and the coefficients $f_k, g_k, h_k$ are modular forms of weight $k$ for $\mathrm{\tilde O}(T)$, $T = H^{\oplus 2} \oplus A_6(-1)$. 
The coefficient $h_5$ is a Borcherds product of weight $5$ with zeros exactly on the Heegner divisors of discriminant $-3/7$ that cut out the locus of K3 surfaces whose very general transcendental lattice is $H^{\oplus 2} \oplus A_5(-1)$. From Koecher's principle and the computations in the previous subsection, we obtain the coefficients $$f_4 = -3 F_4, \; f_6 = 2F_6 + 414720 \chi_6, \; g_6 = -207360 \chi_6,$$ $$h_7 = 12^{9/2} \chi_7, \quad g_8 = -\frac{12^5}{2} \chi_8,$$ while all remaining coefficients are determined only up to a multiple of $\chi_5$. \\ 

The remaining coefficients $g_{10}, g_{12}, h_9$ can be determined using the following trick: We again specialize to the case that the $x^1$- and $x^0$-coefficients have a common root: $$y^2 = x^3 + (u^3 + f_4 u + f_6) x^2 + (u - \beta_2) (g_6 u^2 + \alpha_8 u + \alpha_{10}) + (u - \beta_2)^2 (h_5 u + \gamma_7)^2,$$ where $f_4, f_6, \beta_2, g_6, \alpha_8, \alpha_{10}, h_5, \gamma_7$ now generate the graded ring of meromorphic modular forms attached to the reducible root system $A_1 + A_4$ in the sense of \cite{WW21}. By comparison with the previous subsection, we find that the root $\beta_2$ is again a meromorphic Gritsenko lift: $$\beta_2 = 12 \cdot \mathrm{Lift}(\phi_{-2}^{(1)} \otimes E_4),$$ where $E_4$ is now the Jacobi Eisenstein series of $A_4$ lattice index. The coefficients then turn out to be uniquely determined by the identities $$\iota^* g_6 \cdot \beta_2^3 + \iota^* g_8 \cdot \beta_2^2 + \iota^* g_{10} \cdot \beta_2 + \iota^* g_{12} = 0, \quad \iota^* h_5 \cdot \beta_2^2 + \iota^* h_7 \cdot \beta_2 + \iota^* h_9 = 0.$$

Altogether, we have the following coefficients:

\begin{align*}
f_4 &= -3 F_4; \\
h_5 &= \frac{5}{2} 12^{9/2} \chi_5; \\
f_6 &= 2 F_6 + 414720 \chi_6; \\
g_6 &= -207360 \chi_6; \\
h_7 &= 12^{9/2} \chi_7; \\
g_8 &= -\frac{12^5}{2} \chi_8; \\
h_9 &= 12^{9/2} \chi_9 - \frac{3}{2} 12^{9/2} F_4 \chi_5 \\ &= \frac{12^{5/2}}{2} \cdot \mathrm{Lift} \Big( q^{1/7} (e_{3\ell} - e_{-3\ell}) + 11q^{2/7} (e_{2\ell} - e_{-2\ell}) + 109q^{4/7} (e_{\ell} - e_{-\ell}) \pm ... \Big); \\
g_{10} &= -12^5 \chi_{10} + 2^9 3^6 \cdot F_4 \chi_6 + 2^{18} \cdot 3^9 \cdot 5^3 \cdot \chi_5^2 \\ &= -2592 \cdot \mathrm{Lift} \Big( q^{1/7} (e_{3\ell} + e_{-3\ell}) + 5q^{2/7} (e_{2\ell} + e_{-2\ell}) + 13q^{4/7} (e_{\ell} + e_{-\ell}) - 462q e_0 \pm ... \Big); \\
g_{12} &= 12^5 \chi_{12} + \frac{12^5}{2} F_4 \chi_8 + 2^{19} \cdot 3^9 \cdot 37 \cdot \chi_5 \chi_7 - 2^{11} \cdot 3^4 \cdot F_6 \chi_6 - 2^{18} \cdot 3^{11} \cdot 5 \chi_6^2 \\ &= 2^{19} \cdot 3^9 \cdot 5 \chi_5 \chi_7 - 2^{17} \cdot 3^9 \cdot 5^2 \chi_6^2  \\ &\quad \quad + 864 \cdot \mathrm{Lift} \Big( q^{1/7} (e_{3\ell} + e_{-3\ell}) + 9q^{2/7} (e_{2\ell} + e_{-2\ell}) + 89q^{4/7} (e_{\ell} + e_{-\ell}) + 1762 q e_0 - 519q^{8/7} (e_{3\ell} + e_{-3\ell}) \pm ... \Big). \end{align*}

All coefficients other than $g_{12}$ can be represented as Gritsenko lifts.
\subsection{\texorpdfstring{$n=7$}{n equals 7}}
The Weierstrass equation is $$y^2 = x^3 + (u^3 + f_4 u + f_6) x^2 + (g_4 u^4 + g_6 u^3 + g_8 u^2 + g_{10} u + g_{12}) x + (h_5 u^2 + h_7 u + h_9)^2$$ and the coefficients $f_k, g_k, h_k$ are modular forms of weight $k$ for $\mathrm{\tilde O}(T)$, $T = H^{\oplus 2} \oplus A_7(-1)$.
Using the same argument as the previous subsections, we find that $g_4$ is a Borcherds product with zeros exactly on the Heegner divisors of discriminant $-7/16$ and that the coefficients $f_6, g_6, h_5, h_7$ are immediately determined uniquely by their specialization to $g_4 = 0$: $$f_6 = 2 F_6 + 414720 \chi_6, \; g_6 = -207360 \chi_6, \; h_5 = \frac{5}{2} 12^{9/2} \chi_5, \; h_7 = 12^{9/2} \chi_7,$$ while the other coefficients are determined only up to multiples of $g_4$. \\

We proceed by first specializing to the case that the $x^1-$ and $x^0-$coefficients of the Weierstrass polynomial share a common root $\beta_2$: \begin{equation}\label{eq:A1+A5}y^2 = x^3 + (u^3 + f_4 u + f_6) x^2 + (u - \beta_2) (g_4 u^3 + \alpha_6 u^2 + \alpha_8 u + \alpha_{10}) + (u - \beta_2)^2 (h_5 u + \gamma_7)^2,\end{equation}
bearing in mind that this corresponds to the vanishing locus of a resultant of degree $44$ and that (by another appeal to Koecher's principle) all coefficients $f_k, g_k, h_k$, being of weight less than $44$, can be uniquely determined from their specializations. \\

We can further specialize Equation \eqref{eq:A1+A5} by imposing the existence of a section with $x(u) = -g_4 (u - \beta_2)$. Then $(u - \beta_2)^2$ can be factored out of the right-hand side of \eqref{eq:A1+A5}, and the prefactor $-g_4$ makes the remainder become quadratic: \begin{align*} &\quad x^3 + (u^3 + f_4 u + f_6) x^2 + (u - \beta_2) (g_4 u^3 + \alpha_6 u^2 + \alpha_8 u + \alpha_{10}) + (u - \beta_2)^2 (h_5 u + \gamma_7)^2 \\ &= (u - \beta_2)^2 q(u), \quad \mathrm{deg}(q) = 2.
\end{align*}
The discriminant of $q$ splits into $g_4$ times \begin{align*} R_{20} &:=  -f_4^2g_4^3 + 2f_4g_4^4 - g_4^5 - 4g_4^3\alpha_6\beta_2 + 4g_4^2h_5^2\beta_2 - 4f_6g_4^2\alpha_6 + 2f_4g_4^2\alpha_8 - 2g_4^3\alpha_8 \\ &+ 4f_6g_4h_5^2 - 4f_4g_4h_5\gamma_7 + 4g_4^2h_5\gamma_7 - g_4\alpha_8^2 + 4g_4\alpha_6\alpha_{10} - 4\alpha_{10}h_5^2 + 4\alpha_8h_5\gamma_7 - 4\alpha6\gamma_7^2; \end{align*}
hence $R_{20}$ is holomorphic modular form whose only zero outside of the poles of $\beta_2$ is the Noether--Lefschetz divisor associated to the additional infinite-order section $$x(u) = -g_4(u - \beta_2), \quad y = (u - \beta_2) \sqrt{q(u)}.$$ Indeed, $R_{20}$ is a Borcherds product of weight 20, with sixth-order zeros on the Heegner divisors of discriminant $-1/4$ and simple zeros on the Heegner divisors of discriminant $-3/4$, and we can compare Fourier coefficients between $R_{20}$ and its Borcherds product expansion. This turns out \emph{not} to determine all coefficients uniquely, but it does determine them after using the additional constraints that certain expressions (e.g. $g_6 = \alpha_6 - \beta_2 g_4$) arise as restrictions of modular forms on $\mathrm{\tilde O}(T)$. We finally obtain the following coefficients: \\

\begin{align*} f_4 &= -3 F_4 - 2903040 \chi_4 \\ &= -144 \cdot \mathrm{Lift} \Big( 5e_0 + e_{4\ell} + 4q^{1/16} (e_{3\ell} + e_{-3\ell}) + 6q^{1/4} (e_{2\ell} + e_{-2\ell}) + 4q^{9/16}(e_{\ell} + e_{-\ell}) + 2q e_0 + 10q e_{4\ell} + ... \Big); \\ g_4 &= -725760 \chi_4 \\ &= 144 \cdot \mathrm{Lift} \Big( e_{4\ell} - q^{1/16} (e_{3\ell} + e_{-3\ell}) + q^{1/4} (e_{2\ell} - e_{-2\ell}) - q^{9/16} (e_{\ell} - e_{-\ell}) + 2q e_0 \pm ... \Big); \\ h_5 &= \frac{5}{2} 12^{9/2} \chi_5 \\ &= -144\sqrt{3} \cdot \mathrm{Lift} \Big( q^{1/16} (e_{3\ell} - e_{-3\ell}) - 2q^{1/4} (e_{2\ell} - e_{-2\ell}) + 3q^{9/16} (e_{\ell} - e_{-\ell}) - 5q^{25/16} (e_{\ell} - e_{-\ell}) \pm ... \Big); \\ f_6 &= 2 F_6 + 414720 \chi_6 \\ &= 144 \cdot \mathrm{Lift} \Big( -7e_0 + e_{4\ell} + 4q^{1/16} (e_{3\ell} + e_{-3\ell}) + 30q^{1/4} (e_{2\ell} + e_{-2\ell}) + 100q^{9/16} (e_{\ell} + e_{-\ell}) + 122 q e_0 \pm ... \Big); \\ g_6 &= -207360 \chi_6 \\ &= -288 \cdot \mathrm{Lift} \Big( 2 e_{4\ell} + q^{1/16} (e_{3\ell} + e_{-3\ell}) - 10q^{1/4}(e_{2\ell} + e_{-2\ell}) + 25q^{9/16} (e_{\ell} + e_{-\ell}) - 92 qe_0 - 48q e_{4\ell} \pm ... \Big); \\ h_7 &= 12^{9/2} \chi_7 \\ &= 288\sqrt{3} \cdot \mathrm{Lift}\Big( q^{1/16} (e_{3\ell} - e_{-3\ell}) + 4q^{1/4} (e_{2\ell} + e_{-2\ell}) - 21q^{9/16} (e_{\ell} + e_{-\ell}) - 48q^{17/16} (e_{3\ell} - e_{-3\ell}) \pm ... \Big); \\ g_8 &= -\frac{12^5}{2} \chi_8 + 12^5 \cdot 7 \cdot \chi_4 F_4 - 12^8 \cdot 5^2 \cdot 7^2 \cdot \chi_4^2 \\ &= -3\chi_4^2 + 864 \cdot \mathrm{Lift} \Big( e_{4\ell} + 2q^{1/16}(e_{3\ell} + e_{-3\ell}) + q^{1/4} (e_{2\ell} + e_{-2\ell}) - 46q^{9/16} (e_{\ell} + e_{-\ell}) + 386qe_0 - 144q e_{4\ell} \pm ... \Big); \\ h_9 &= 12^{9/2} \chi_9 - \frac{3}{2} \cdot 12^{9/2} \chi_5 F_4 - 12^{17/2} \cdot 5^2 \cdot 7 \chi_4 \chi_5 \\ &= -12^{17/2} \cdot \frac{175}{2} \chi_4 \chi_5 - 72 \cdot \mathrm{Lift}\Big( q^{1/16} (e_{3\ell} - e_{-3\ell}) + 10q^{1/4} (e_{2\ell} - e_{-2\ell}) + 99q^{9/16} (e_{\ell} - e_{-\ell}) \pm ...  \Big);\\ g_{10} &= -12^5 \chi_{10} + 12^9 \cdot 5^3 \cdot \chi_5^2 - 2^{20} \cdot 3^8 \cdot 5^2 \cdot 7 \cdot \chi_4 \chi_6 - 2^{11} \cdot 3^4 \cdot 5 \chi_4 F_6 + 2^9 \cdot 3^6 \chi_6 F_4 \\ &= 12^9 \cdot \frac{175}{2} (\chi_5^2 - \chi_4 \chi_6) \\ &\quad - 576 \cdot \mathrm{Lift}\Big( 2e_{4\ell} + 7q^{1/16} (e_{3\ell} + e_{-3\ell}) + 38q^{1/4} (e_{2\ell} + e_{-2\ell}) + 79 q^{9/16} (e_{\ell} + e_{-\ell})-4316qe_0 - 720q e_{4\ell} \pm ... \Big)  \\ g_{12} &= 12^5 \chi_{12} - 2^{18} \cdot 3^{11} \cdot 5 \chi_6^2 + 2^{19} \cdot 3^9 \cdot 37 \chi_5 \chi_7 - 2^{20} \cdot 3^8 \cdot 43 \chi_4 \chi_8 + 2^9 \cdot 3^5 \cdot F_4 \chi_8 - 2^{11} \cdot 3^4 \cdot F_6 \chi_6 \\ & \quad - 2^{16} \cdot 3^9 \cdot 7 \cdot 2963 F_4 \chi_4^2  - 2^8 \cdot 3^6 F_4^2 \chi_4 - 2^{27} \cdot 3^{12} \cdot 5^2 \cdot 7^2 \cdot 11 \cdot 23 \chi_4^3 \\ &= 2^{19} \cdot 3^9 \cdot 5 \chi_5 \chi_7 - 2^{17} \cdot 3^9 \cdot 5^2 \chi_6^2 + 2^{16} \cdot 3^9 \cdot 5^2 \cdot 7 \chi_4 \chi_8 - 2^{15} \cdot 3^9 \cdot 5^2 \cdot 7^3 F_4 \chi_4^2 - 2^{23} \cdot 3^{11} \cdot 5^3 \cdot 7^4 \chi_4^3 \\ &+ 144 \cdot \mathrm{Lift} \Big( e_{4\ell} + 5q^{1/16} (e_{3\ell} + e_{-3\ell}) + 49q^{1/4} (e_{2\ell} + e_{-2\ell}) + 485 q^{9/16} (e_{\ell} + e_{-\ell}) + 9602 q e_0 - 672 q e_{4\ell} \pm ... \Big). \end{align*}

The coefficients $f_4, g_4, h_5, f_6, g_6, h_7$ admit representations as Gritsenko lifts, while the coefficients $g_8, h_9, g_{10}, g_{12}$ do not. The vector-valued modular forms inside Lift in the expressions above determine the leading Fourier--Jacobi coefficients of $g_k$ and $h_k$.
\end{proof}
\section{The \texorpdfstring{$D_n$}{Dn}-family of K3 surfaces}

\begin{proof}[\textbf{Proof of Theorem~\ref{thm:main}\,\ref{thm:partI}}]

In Proposition \ref{prop:WEQ_rank10} we established a Weierstrass normal form for the general $S$-polarized K3 surface with $T = S^{\perp} = H \oplus H \oplus D_8(-1)$.  The fact that the multiplicity of the frame $(D_8, \{ \mathbb{I}\})$ was shown to equal one implies that, for very general parameters, the associated Jacobian elliptic fibration is uniquely determined by its frame, up to automorphisms of the surface.  It follows that the parameter space of the Weierstrass models coincides with the moduli space associated with the discriminant kernel $\mathrm{\tilde O}(T)$. Thus, the coefficients $F_k$, $G_k$, and $H_k$ define modular forms for $\mathrm{\tilde O}(T)$ of the indicated weights.

We set $H_8=-G_4^2$ and write the Weierstrass equation as
\begin{equation}
\label{eq:D8}
\begin{split}
 y^2  = x^3 + & \Big( u^3 + F_4 u + F_6\Big) x^2 + \Big( G_8 u^2 + G_{10} u + G_{12} \Big) x \\
 - & \ G_4^2  u^5 + H_{10} u^4 + H_{12} u^3 + H_{14} u^2 + H_{16} u + H_{18}.
\end{split}
\end{equation}
We now explain how the coefficients $F_k, G_k, H_k$ are determined.
\subsection*{Reduction to \texorpdfstring{$A_7$}{A7}} The sublattice $H^{\oplus 2} \oplus A_7(-1)$ embeds in $H^{\oplus 2} \oplus D_8(-1)$ in two $\mathrm{\tilde O}(T)$-inequivalent ways as a Heegner divisor of discriminant $-1$. Both occur as the divisor of a Borcherds product of weight $128$. This implies abstractly that there is a specialization map from \eqref{eq:D8} to the family of K3 surfaces with transcendental lattice $H^{\oplus 2} \oplus A_7(-1)$, and moreover that the coefficients are uniquely determined by their images under this specialization. We will now make this concrete. \\ 

Let $h_5$ be a variable independent of the coefficients of \eqref{eq:D8}. After applying the substitution $$x \mapsto x \pm \left( G_4 u + \frac{h_5^2 - G_{10}}{2G_4} \right),$$ we find that Vinberg's $D_8$ model coincides with the $A_7$ model if and only if the coefficients $F_k, G_k, H_k$ and $f_k, g_k, h_k$ are related by

\begin{align} 
    F_4 &= f_4 -\frac{3}{2} g_4;  \notag\\
    G_4 &= \pm \frac{1}{2} g_4;  \notag\\
    F_6 &= f_6 - \frac{3}{2} g_6;  \notag\\
    G_8 &= -f_4 g_4 + \frac{3}{4} g_4^2 + g_8;  \notag\\
    G_{10} &= -f_4 g_6 - f_6 g_4 + \frac{3}{2} g_4 g_6 + g_{10};  \notag\\
\label{eq:D8toA7}
    H_{10} &= -\frac{1}{2} g_4 g_6 + h_5^2; \\
    G_{12} &= -f_6 g_6 + \frac{3}{4} g_6^2 + g_{12};  \notag\\
    H_{12} &= \frac{1}{4} f_4 g_4^2 - \frac{1}{8} g_4^3 - \frac{1}{2} g_4 g_8 + 2 h_5 h_7 - \frac{1}{4} g_6^2;  \notag\\
    H_{14} &= \frac{1}{2} f_4 g_4 g_6 - \frac{3}{8} g_4^2 g_6 + \frac{1}{4} f_6 g_4^2 - \frac{1}{2} g_4 g_{10} + 2 h_5 h_9 - \frac{1}{2} g_6 g_8 + h_7^2;  \notag\\
    H_{16} &= \frac{1}{4} f_4 g_6^2 - \frac{3}{8} g_4 g_6^2 + \frac{1}{2} f_6 g_4 g_6 - \frac{1}{2} g_4 g_{12} - \frac{1}{2} g_6 g_{10} + 2 h_7 h_9;  \notag\\
    H_{18} &= \frac{1}{4} f_6 g_6^2 - \frac{1}{8} g_6^3 - \frac{1}{2} g_6 g_{12} + h_9^2. \notag
\end{align}

Conversely, for any values of $f_4, f_6, g_4, g_6, g_8, g_{10}, g_{12}, h_7, h_9$ and $H_{16}, H_{18}$, we can solve Equation \eqref{eq:D8toA7} with unique values for $F_4, F_6, G_4^2, G_8, G_{10}, G_{12}, H_{10}, H_{12}, H_{14}$ and $h_5$. \\

As expected, Equation \eqref{eq:D8toA7} turns out to have a unique (up to a sign in the case of $G_4$) solution that is compatible with our earlier computations for the $H^{\oplus 2} \oplus A_7(-1)$ family. It is easily computed using the compatibility of the Gritsenko lift with the pullback $\iota^*$. To express the results, we note that the discriminant form $D_8^{\vee}/D_8$ is just $\mathbb{Z}^2 / 2 \mathbb{Z}^2$ with the quadratic form $(x,y) \mapsto \frac{1}{2}xy \in \mathbb{Q}/\mathbb{Z}$. We label the cosets $\mathbf{0}, \mathbf{1},\mathbf{2}, v$ where $\mathbf{0}$ is the zero element, $\mathbf{1}, \mathbf{2}$ are isotropic, and $v$ has norm $1/2+\mathbb{Z}$. The cosets $\mathbf{1}$ and $\mathbf{2}$ are equivalent under $\mathrm{O}^+(H^{\oplus 2} \oplus D_8(-1))$ and are effectively indistinguishable.\\

The coefficients of \eqref{eq:D8} are determined as follows:

 \begin{align*} F_4 &= -360 \cdot \mathrm{Lift}(2 e_{\mathbf{0}} + e_{\mathbf{1}} + e_{\mathbf{2}});\\
 G_4 &= \pm 72 \cdot \mathrm{Lift}(e_{\mathbf{1}} - e_{\mathbf{2}}); \\
 F_6 &= 1008 \cdot \mathrm{Lift} \left( \begin{gathered} \quad (1 + 24q + 24q^2 + ...) \cdot (-e_{\mathbf{0}} + e_{\mathbf{1}} + e_{\mathbf{2}}) \\ + 24 (q^{1/2} + 4q^{3/2} + 6q^{5/2} + ...) e_{v} \end{gathered} \right); \\
 G_8 &= -15 G_4^2 + 1296 \Phi_8, \\ &\quad \Phi_8 = \mathrm{Lift} \left( \begin{gathered} \; (256q + 2048q^2 + 7168q^3 + ...) e_{\mathbf{0}} \\ + (1 - 16q + 112q^2 - 448q^3 \pm ...) (e_{\mathbf{1}} + e_{\mathbf{2}}) \\ + (-32q^{1/2} - 896q^{3/2} - 4032q^{5/2} \pm ...) e_v \end{gathered} \right); \\  G_{10} &= \frac{24}{35} F_4 F_6 - \frac{864}{35} \Phi_{10}, \\ &\quad \Phi_{10} = \mathrm{Lift} \left( \begin{gathered} \; (44 - 11424q - 387744q^2 \pm ...) e_{\mathbf{0}} \\ + (1 - 11256 q - 360696 q^2 \pm ...) (e_{\mathbf{1}} + e_{\mathbf{2}}) \\ + (-336q^{1/2} - 81984 q^{3/2} - 1050336 q^{5/2} \pm ...) e_v \end{gathered} \right); \\ H_{10} &= \frac{1}{35} F_4 F_6 - \frac{144}{35} \Phi_{10}, \\ &\quad \Phi_{10} = \mathrm{Lift} \left( \begin{gathered} (11 + 10584 q + 333144 q^2 \pm ...) e_{\mathbf{0}} \\ + (-26 - 3024q -83664 q^2 \pm ...) (e_{\mathbf{1}} + e_{\mathbf{2}}) \\ + (-504 q^{1/2} - 122976q^{3/2} - 1575504q^{5/2} \pm ...) e_v \end{gathered}\right); \\ G_{12} &= -\frac{6}{125} F_4^3 + \frac{22}{5} F_4 G_4^2 + \frac{9}{49} F_6^2 - \frac{2592}{175} \Phi_{12}, \\ &\quad \Phi_{12} = \mathrm{Lift} \left( \begin{gathered} \; (13 - 35360 q + 318560q^2 \pm ...) e_{\mathbf{0}} \\ + (-6 + 38720 q + 114880q^2 \pm ...) (e_{\mathbf{1}} + e_{\mathbf{2}}) \\ + (-4400q^{1/2} + 49600q^{3/2} + 4610400q^{5/2} \pm ...) e_v \end{gathered} \right); \\ H_{12} &= -\frac{2}{375} F_4^3 - \frac{27}{5} F_4 G_4^2 + \frac{1}{49} F_6^2 - \frac{144}{175} \Phi_{12},\\ &\quad \Phi_{12} = \mathrm{Lift} \left( \begin{gathered} \; (26 + 108480q + 10672320q^2 \pm ...) e_{\mathbf{0}} \\ + (163 - 17760q + 1030560q^2 \pm ...) (e_{\mathbf{1}} + e_{\mathbf{2}}) \\ + (2400q^{1/2} - 1200000q^{3/2} - 44740800q^{5/2} \pm ...) e_v \end{gathered}\right); \\ H_{14} &= -\frac{1}{50} F_4 G_{10} + \frac{122}{25} F_4 H_{10} - \frac{98}{5} G_4^2 F_6 + \frac{3888}{25} \mathrm{Lift}(\Phi_{14}), \\ &\quad \Phi_{14} = \mathrm{Lift} \left( \begin{gathered} \; (-512 q - 139264q^2 - 5163008q^3 \pm ...) e_{\mathbf{0}} \\ +(1 + 248q + 3832q^2 + \pm...) (e_{\mathbf{1}} + e_{\mathbf{2}}) \\ + (16q^{1/2} + 7744q^{3/2} + 1021536q^{5/2} \pm ...) e_v \end{gathered}\right); \\ H_{16} &= \frac{497}{62500} F_4^4 + \frac{1591}{1250} F_4^2 G_4^2 + \frac{61}{20} G_4^4 - \frac{639}{9625} F_6^2 F_4 + \frac{2}{25} G_{12} F_4 - \frac{62}{125} H_{12} F_4 \\ &\quad + \frac{332}{1925} G_{10} F_6 + \frac{1664}{1925} H_{10} F_6 - \frac{648}{240625} \Phi_{16}, \\ &\quad \Phi_{16} = \mathrm{Lift} \left( \begin{gathered} \; (2414 + 217034080q - 16134309280q^2 \pm ...) e_{\mathbf{0}} \\ + (1007 + 175093040q + 7395690160q^2 \pm ...) (e_{\mathbf{1}} + e_{\mathbf{2}}) \\  + (-5547200 q^{1/2} - 1396256000q^{3/2} - 26340777600q^{5/2} \pm ...) e_v \end{gathered} \right); \\ H_{18} &= \frac{8328}{175} F_4 F_6 G_4^2 - \frac{174}{5} G_8 H_{10} + \frac{7}{5} F_4 H_{14} + \frac{137}{35} F_6 H_{12} - \frac{2}{35} F_6 G_{12} \\ &- \frac{1}{250} F_4^2 G_{10} - \frac{793}{125} F_4^2 H_{10} + \frac{1797}{10} G_4^2 G_{10} - \frac{2319}{5} G_4^2 H_{10} + \frac{1296}{875} \Phi_{18}, \\ &\quad \Phi_{18} = \mathrm{Lift} \left( \begin{gathered} \; (134656q - 29360128q^2 - 1986553856q^3 \pm ...) e_{\mathbf{0}} \\ + (37 - 135544q + 22084744q^2 \pm ...) (e_{\mathbf{1}} + e_{\mathbf{2}}) \\ + (5392 q^{1/2} - 4697792 q^{3/2} + 53266272q^{5/2} \pm ...) e_v \end{gathered} \right).\end{align*}
\end{proof}

 \begin{corollary}
 The modular forms described by the coefficients $F_k$, $G_k$, $H_k$ have Fourier--Jacobi expansions beginning in degrees $d = 0, 1, 2$, respectively. Dividing the leading coefficients in the Fourier--Jacobi series by $\Delta^d$ yields generators of the Wirthm\"uller ring of Weyl-invariant weak Jacobi forms for $D_8$.
 \end{corollary}

\begin{remark}
Lowering the degree of the constant term (with respect to $x$) in the Weierstrass equation~\eqref{eq:D8} for the Vinberg family of  $H \oplus D_8(-1)$-polarized K3 surfaces yields the families with transcendental lattices  $H^{\oplus 2} \oplus D_n(-1)$ for $4 \le n \le 7$; see Remarks \ref{rem:11-14}, \ref{rem:15}, \ref{rem:16}. For these smaller families, the expressions for the coefficients simplify considerably. We determine them explicitly in Appendix~\ref{sec:coefficients}.
\end{remark}
\appendix
\section{Pullbacks of the Vinberg family}
\label{sec:coefficients}
Lowering the degree of the constant term (with respect to $x$) in the Weierstrass equation~\eqref{prop:WEQ_rank10} in Proposition \ref{prop:WEQ_rank10} for the Vinberg family of  $H \oplus D_8(-1)$-polarized K3 surfaces yields the families with transcendental lattice $H^{\oplus 2} \oplus D_n(-1)$ for $4 \le n \le 7$. For these smaller families, the expressions we found for the coefficients simplify considerably. \\

(We omit the case $n=3$, because $D_3 = A_3$ appeared in the $A_n$-family in \cite{MR4987329}.)
\subsection{\texorpdfstring{$n=4$}{n equals 4}} 
\label{ssec:4}
The coefficients of the Weierstrass model $$y^2 = x^3 + (u^3 + F_4 u + F_6) x^2 + (G_8 u^2 + G_{10} u + G_{12}) x + H_{16} u + H_{18}$$ are Gritsenko lifts on the discriminant kernel of $T = H^{\oplus 2} \oplus D_4(-1)$. \\

Recall that $D_4^{\vee}/D_4 \cong \mathbb{Z}^2 / 2 \mathbb{Z}^2$ has size four. We label the cosets $0$ and $v_1, v_2, v_3$. The cosets $v_i$ have norm $1/2 + \mathbb{Z}$, and the orthogonal group of the discriminant form $\mathrm{O}(D_4^{\vee}/D_4) \cong S_3$ acts as permutations of them. The cosets $v_1, v_2, v_3$ are therefore essentially indistinguishable. However, most of the coefficients are not  $\mathrm{O}(D_4^{\vee}/D_4)$-invariant; they require the choice of a distinguished coset (say, $v_3$).
The map $\mathrm{O}(D_4) \rightarrow \mathrm{O}(D_4^{\vee}/D_4)$ is not surjective, but its image contains the order three permutations. Let $t$ (for triality) be any element of $\mathrm{O}(D_4)$ that acts as the permutation $$t : (v_1,v_2,v_3) \mapsto (v_2,v_3,v_1),$$ and note that $t$ induces a map $$t^* : M_*(\mathrm{O}^+(T)) \longrightarrow M_*(\mathrm{O}^+(T)), \quad f(Z) \mapsto f(tZ).$$

The coefficients are given as follows:

\begin{align*} F_4 &= -3 E_4 = -720 \cdot \mathrm{Lift} \left( \begin{gathered}  \; (1 + 24q + 24q^2 + 96q^3 + 24q^4 + ...) e_0 \\ +(8q^{1/2} + 32q^{3/2} + 48q^{5/2} + 64q^{7/2} + ...) (e_{v_1} + e_{v_2} + e_{v_3}). \end{gathered} \right); \\
F_6 &= 2 E_6 = -1008 \cdot \mathrm{Lift} \left( \begin{gathered}  \; (1 - 144q - 912q^2 - 4032q^3 - 7056q^4 - ...) e_0 \\ +(-16q^{1/2} - 448q^{3/2} - 2016q^{5/2} - 5504q^{7/2} - ...) (e_{v_1} + e_{v_2} + e_{v_3}). \end{gathered} \right); \\
G_8 &= 12^4 \cdot \mathrm{Lift} \left( \begin{gathered} \; 0 \cdot e_0 \\ + (q^{1/2} - 12q^{3/2} + 54q^{5/2} - 88q^{7/2} \pm ...) (e_{v_1} + e_{v_2}) \\ + (-2q^{1/2} + 24q^{3/2} - 108 q^{5/2} + 176q^{7/2} \pm ...) e_{v_3} \end{gathered} \right);\\
G_{10} &= -2 \cdot 12^4 \cdot \mathrm{Lift}\left( \begin{gathered} \; (-24q + 192q^2 - 288q^3 - 1536q^4 + 5040q^5 \pm ...) e_0 \\ + (q^{1/2} + 12q^{3/2} - 210q^{5/2} + 1016q^{7/2} \pm ...) (e_{v_1} + e_{v_2} + e_{v_3}) \end{gathered} \right);\\
G_{12} &= 12^4 \cdot \mathrm{Lift}\left( \begin{gathered} \; (96q + 1536q^2 - 14976q^3 + 24576q^4 + 83520q^5 \pm ...) e_0 \\ + (q^{1/2} - 540q^{3/2} + 2406q^{5/2} + 4424q^{7/2} \pm ...) (e_{v_1} + e_{v_2}) \\ + (4q^{1/2} + 144q^{3/2} + 408q^{5/2} - 14560q^{7/2} \pm ...) e_{v_3} \end{gathered} \right);\\
H_{16} &= 12^8 \Phi_{1,3} \Phi_{2,3}, \\ &\quad \Phi_{i,j} = \mathrm{Lift} \Big( (q^{1/2} - 12q^{3/2} + 54q^{5/2} - 88q^{7/2} \pm ...) (e_{v_i} - e_{v_j}) \Big);\\
H_{18} &= \frac{1}{150}F_4^2 G_{10} + \frac{1}{3} G_8 G_{10} + \frac{62208}{25}\Phi_{18}, \\ &\quad \Phi_{18} = \mathrm{Lift} \left( \begin{gathered}  \; (-24q - 11328q^2 - 1394208q^3 - 13456896q^4 - ...) e_0 \\ +(q^{1/2} + 492q^{3/2} + 67470q^{5/2} + 1693496q^{7/2} + ...) (e_{v_1} + e_{v_2} + e_{v_3}). \end{gathered} \right). \end{align*}

There are three Borcherds products $\psi_1, \psi_2, \psi_3$ of weight $8$, with zeros on the Heegner divisors $r^{\perp}$ of discriminant $-1/2$ and class $r \in v_i + D_4$, also represented by (additive) Gritsenko lifts $$\psi_1 = 12^4 \Phi_{2,3}, \; \psi_2 = 12^4 \Phi_{3,1}, \psi_3 = 12^4 \Phi_{1, 2},$$ and satisfying $\psi_1 + \psi_2 + \psi_3 = 0,$ and the group $\mathrm{O}(D_4^{\vee}/D_4)$ naturally acts on them by permutations. We can express $$G_8 = \psi_1 - \psi_2$$ and $$G_{12} = E_4 (\psi_2 - \psi_1) + 2 \cdot 12^4 \Phi_{12}, \quad \Phi_{12} = \mathrm{Lift}\left( \begin{gathered} \; (48q + 768q^2 - 7488q^3 \pm ...) e_0 \\ + (q^{1/2} - 156q^{3/2} + 870q^{5/2} \pm ...) (e_{v_1} + e_{v_2} + e_{v_3}) \end{gathered} \right);$$ $$H_{16} = -\psi_1 \psi_2;$$

The model therefore becomes $$y^2 = x^3 + (u^3 - 3E_4 u + 2E_6) x^2 + \Big((\psi_1 - \psi_2) (u^2 - E_4) + G_{10} u + \tilde G_{12}) x - \psi_1 \psi_2 u + \frac{1}{3} (\psi_1 - \psi_2) G_{10} + \tilde H_{18}$$ where the remaining coefficients $E_4, E_6, G_{10}, \tilde G_{12}, \tilde H_{18}$ are $t$-invariant.
We obtain a completely $t$-invariant Weierstrass equation by reversing the roles of the fiber and base variables $x,u$ and converting to short Weierstrass form: \begin{align*} Y^2 &= U^3 + (f_4 X^4 + f_{10} X^3 + f_{16} X^2)U + (X^7 + g_6 X^6 + g_{12} X^5 + g_{18} X^3 + g_{24} X^3) \\ &= U^3 + \left(-3E_4 X^4 + G_{10} X^3 + \frac{\psi_1 \psi_2 + \psi_2 \psi_3 + \psi_3 \psi_1}{3} X^2 \right) U \\ &\quad + \left( X^7 + 2 E_6 X^6 + \tilde G_{12} X^5 + \tilde H_{18} X^4 - \frac{(\psi_1 - \psi_2)(\psi_2 - \psi_3)(\psi_3 - \psi_1)}{27} X^3 \right).\end{align*}
The coefficients $f_4, f_{10}, f_{16}, g_6, g_{12}, g_{18}, g_{24}$ of the above Jacobian fibration over $\mathbb{P}^1_{(X)}$ generate the free algebra of modular forms for the full group $\mathrm{O}^+(H^{\oplus 2} \oplus D_4(-1))$. This is the free algebra associated to the $F_4$ root system as described in \cite{WW20}.
\subsection{\texorpdfstring{$n=5$}{n equals 5}} 
The discriminant form $D_5^{\vee}/D_5$ is cyclic, generated by a coset $\ell$ of norm $3/8 + \mathbb{Z}$. 

\begin{align*}
F_4 &= -720 \cdot \mathrm{Lift}\left( \begin{gathered} \; (1 + 12q + 6q^2 + 24q^3 + ...) e_0 \\ + (4q^{3/8} + 12q^{11/8} + 12q^{19/8} + ...) (e_{\ell} + e_{-\ell}) \\ + (6 q^{1/2} + 8q^{3/2} + 24q^{5/2} + ...) e_{2\ell} \end{gathered} \right); \\
F_6 &= -1008 \cdot \mathrm{Lift} \left( \begin{gathered} \; (1 - 108q - 450q^2 - 1656q^3 - ...) e_0 \\ + (-8q^{3/8} - 216q^{11/8} - 792q^{19/8} - ...) (e_{\ell} + e_{-\ell}) \\ + (-18q^{1/2} - 232q^{3/2} - 1080q^{5/2} - ...) e_{2\ell}\end{gathered} \right); \\
G_8 &= \frac{12^4}{2} \cdot \mathrm{Lift}\left( \begin{gathered} \; (8q - 64q^2 + 144q^3 + 128q^4 \pm ...) e_0 \\ + (q^{3/8} - 13q^{11/8} + 67q^{19/8} - 156q^{27/8} \pm ...) (e_{\ell} + e_{-\ell}) \\ + (-4q^{1/2} + 32q^{3/2} - 80q^{5/2} + 300q^{9/2} \pm ...) e_{2\ell} \end{gathered} \right);  \\
G_{10} &= -12^4 \cdot \mathrm{Lift} \left( \begin{gathered} \; (-52q + 128q^2 + 888q^3 - 3328q^4 \pm ...) e_0 \\ +(q^{3/8} + 11q^{11/8} - 221q^{19/8} + 1236 q^{27/8} \pm ...) (e_{\ell} + e_{-\ell}) \\ +(2q^{1/2} + 128q^{3/2} - 680q^{5/2} + 3930q^{9/2} \pm ...) e_{2\ell} \end{gathered} \right); \\
G_{12} &= \frac{12^4}{2} \cdot \mathrm{Lift}\left( \begin{gathered} \; (176q + 3200q^2 - 19104q^3 + 24320q^4 \pm ...) e_0 \\ +(q^{3/8} - 541q^{11/8} + 2947q^{19/8} + 1476q^{27/8} \pm ...) (e_{\ell} + e_{-\ell}) \\ + (8q^{1/2} - 64q^{3/2} - 5600q^{5/2} + 9216q^{7/2} \pm ...) e_{2\ell} \end{gathered} \right); \\ H_{14} &= \frac{12^7}{4} \Phi_7^2, \quad \Phi_7 = \mathrm{Lift} \left( (q^{3/8} - 9q^{11/8} + 27q^{19/8} - 12q^{27/8} \pm ...) (e_{\ell} - e_{-\ell})\right); \\
H_{16} &= \frac{1}{75} F_4^2 G_8 + \frac{1}{6} G_8^2 - \frac{31104}{25} \mathrm{Lift} \left( \begin{gathered} \; (8q + 3776q^2 + 464784q^3 + 4508288q^4 \pm ...) e_0 \\ + (q^{3/8} + 467q^{11/8} + 55747q^{19/8} + 277284q^{27/8} \pm ...) (e_{\ell} + e_{-\ell}) \\ + (-4q^{1/2} - 1888q^{3/2} - 232400q^{5/2} - 2257920q^{7/2} \pm ...) e_{2\ell}\end{gathered} \right); \\
H_{18} &= \frac{1}{150} F_4^2 G_{10} - \frac{3}{154} F_4 F_6 G_8 - \frac{1}{132} G_8G_{10} + \frac{62208}{1925} \Phi_{18}, \\ &\quad \Phi_{18} = \mathrm{Lift} \left( \begin{gathered} \; (-2302q - 874432 q^2 - 81573612q^3 - 545799808q^4 \pm ...) e_0 \\ + (q^{3/8} + 29291 q^{11/8} + 7526179q^{19/8} + 192138516q^{27/8} \pm ...) (e_{\ell} + e_{-\ell}) \\ + (227q^{1/2} + 1088q^{3/2} - 12887900q^{5/2} - 244316160q^{7/2} \pm ...) e_{2\ell} \end{gathered} \right).
\end{align*}
\subsection{\texorpdfstring{$n=6$}{n equals 6}} 
The discriminant group $D_6^{\vee}/D_6 = \{0, u_1, u_2, v\}$ is of the form $\mathbb{Z}^2 / 2 \mathbb{Z}^2$; besides $0$, it contains two elements $u_1, u_2$ of norm $1/4 + \mathbb{Z}$ that are indistinguishable under $\mathrm{O}(D_6^{\vee}/D_6)$, and one coset $v$ of norm $1/2 + \mathbb{Z}$.

\begin{align*}
F_4 &= -720 \cdot \mathrm{Lift} \left(\begin{gathered}  \; (1 + 4q + 4q^2 + 4q^4 + ...) e_0 \\ + (2q^{1/4} + 4q^{5/4} + 2q^{9/4} + 4q^{9/4} + ...) (e_{u_1} + e_{u_2}) \\ + (4q^{1/2} + 8q^{5/2} + 4q^{9/2} + ...) e_v\end{gathered} \right); \\
F_6 &= -1008 \cdot \mathrm{Lift} \left( \begin{gathered} \; (1 - 68q - 260q^2 - 480q^3 - 1028q^4 - ...) e_0 \\ + (-4q^{1/4} - 104q^{5/4} - 292q^{9/4} - 680q^{13/4} - ...) (e_{u_1} + e_{u_2}) \\ + (-20q^{1/2} - 96q^{3/2} - 520q^{5/2} - 576q^{7/2} - ...) e_v \end{gathered} \right); \\
G_8 &= 5184 \cdot \mathrm{Lift} \left( \begin{gathered} \; (32q - 128q^2 + 512q^4 \pm ...) e_0 \\ + (q^{1/4} - 14q^{5/4} + 81q^{9/4} - 238q^{13/4} \pm ...) (e_{u_1} + e_{u_2}) \\ + (-8q^{1/2} + 112q^{5/2} - 648q^{9/2} \pm ...) e_v  \end{gathered} \right); \\
G_{10} &= -10368 \cdot \mathrm{Lift} \left( \begin{gathered} \; (-112q - 704q^2 + 1920q^3 + 4352q^4 \pm ...) e_0 \\ + (q^{1/4} + 10q^{5/4} - 231q^{9/4} + 1466q^{13/4} \pm ...) (e_{u_1} + e_{u_2}) \\ + (4q^{1/2} + 480q^{3/2} + 40q^{5/2} - 4800q^{7/2} \pm ..._) e_v \end{gathered} \right); \\
G_{12} &= -\frac{3}{10} F_4 G_8 + \frac{2592}{5} \cdot \mathrm{Lift} \left( \begin{gathered} \; (2912q + 11392q^2 - 192000q^3 + 1091072q^4 \pm ...) e_0 \\ + (q^{1/4} - 7454q^{5/4} + 44961q^{9/4} + 18722q^{13/4} \pm ...) (e_{u_1} + e_{u_2}) \\ + (232q^{1/2} + 9600q^{3/2} - 116528q^{5/2} + 134400q^{7/2} \pm ...) e_v \end{gathered} \right); \\
H_{12} &= -\frac{12^6}{4} \Phi_6^2, \quad \Phi_6 = \mathrm{Lift} \left( (q^{1/4} - 6q^{5/4} + 9q^{9/4} + 10q^{13/4} \pm ...) (e_{u_1} - e_{u_2}) \right); \\
H_{14} &= -\frac{1}{14}F_6 G_8 + \frac{5184}{7} \cdot \mathrm{Lift} \left( \begin{gathered} \; (32q - 16256q^2 - 467712q^3 - 1805824q^4 \pm ...) e_0 \\ + (q^{1/4} - 518q^{5/4} - 9495q^{9/4} + 68810q^{13/4} \pm ...) (e_{u_1} + e_{u_2}) \\ + (-8q^{1/2} + 4032q^{3/2} + 133168q^{5/2} + 927360q^{7/2} \pm ...) e_v\end{gathered} \right); \\
H_{16} &= \frac{1}{6} G_8^2 + \frac{1}{75} F_4^2 G_8 + \frac{19}{5} F_4 H_{12} - \frac{15552}{25} \Phi_{16}, \\ &\quad \Phi_{16} = \mathrm{Lift} \left( \begin{gathered} \; (32q + 15232q^2 + 1920000q^3 + 25682432q^4 \pm ...) e_0 \\ + (q^{1/4} + 466q^{5/4} + 55281q^{9/4} + 222002q^{13/4} \pm ...) (e_{u_1} + e_{u_2}) \\ + (-8q^{1/2} - 3840q^{3/2} - 495248q^{5/2} - 8348160q^{7/2} \pm ...) e_v\end{gathered} \right); \\
H_{18} &= \frac{1}{150} F_4^2 G_{10} + \frac{1}{3} G_8 G_{10} + 13 F_6 H_{12} + 3 F_4 H_{14} + \frac{15552}{25} \Phi_{18}, \\ &\quad \Phi_{18} = \mathrm{Lift} \left( \begin{gathered} (-112 q - 54464q^2 - 7271040q^3 - 160292608q^4 \pm ...) e_0 \\ + (q^{1/4} + 98q^{5/4} + 78249q^{9/4} + 1401266q^{13/4} \pm ...) (e_{u_1} + e_{u_2}) \\ + (4q^{1/2} + 2400q^{3/2} + 478120q^{5/2} + 33936960q^{7/2} \pm ...) e_v\end{gathered} \right).
\end{align*}
\subsection{\texorpdfstring{$n=7$}{n equals 7}} 
The discriminant group $D_7^{\vee}/D_7 = \langle \ell \rangle$ is cyclic, generated by a coset $\ell$ of norm $1/8$.

\begin{align*}
F_4 &= -720 \cdot \mathrm{Lift} \left( \begin{gathered} \; (1 + 2q^2 + 2q^8 + ...) e_0 \\ + (q^{1/8} + q^{9/8} + q^{25/8} + ...) (e_{\ell} + e_{-\ell}) \\ + (2q^{1/2} + 2q^{9/2} + 2q^{25/2} + ...) e_{2\ell}\end{gathered} \right); \\
F_6 &= -1008 \cdot \mathrm{Lift} \left( \begin{gathered} \; (1 - 24q - 166q^2 - 144q^3 - 312q^4 - ...) e_0 \\ + (-2q^{1/8} - 50q^{9/8} - 96q^{17/8} - 242q^{25/8} - ...) (e_{\ell} + e_{-\ell}) \\ + (-22q^{1/2} - 48q^{3/2} - 144q^{5/2} - 192q^{7/2} - ...) e_{2\ell}\end{gathered} \right); \\
G_8 &= 2592 \cdot \mathrm{Lift} \left( \begin{gathered} \; (96 q + 128q^2 - 192q^3 - 768q^4 \pm ...) e_0 \\ + (q^{1/8} - 15q^{9/8} + 96q^{17/8} - 335q^{25/8} \pm ...) (e_{\ell} + e_{-\ell}) \\ + (-16q^{1/2} - 192q^{3/2} + 768q^{7/2} + 240q^{9/2} \pm ...) e_{2\ell}\end{gathered} \right); \\
G_{10} &= \frac{24}{35} F_4 F_6 - \frac{1728}{35} \mathrm{Lift} \left( \begin{gathered} \; (22 - 6048q - 275812q^2 - 2455488q^3 \pm ...) e_0 \\ + (q^{1/8} - 11255q^{9/8} - 371952q^{17/8} - 3107159q^{25/8} \pm ...) (e_{\ell} + e_{-\ell}) \\ + (-124q^{1/2} - 52416q^{3/2} - 913248q^{5/2} - 5693184q^{7/2} \pm ...) e_{2\ell}\end{gathered} \right); \\
H_{10} &= 62208 \Phi_5^2, \quad \Phi_5 = \mathrm{Lift} \left( (q^{1/8} - 3q^{9/8} + 5q^{25/8} - 7q^{49/8} \pm ...) (e_{\ell} - e_{-\ell}) \right); \\
G_{12} &= -\frac{3}{10} F_4 G_8 + \frac{1296}{5} \cdot \mathrm{Lift} \left( \begin{gathered} \; (4896q + 3968q^2 + 90048 q^3 + 2034432q^4 \pm ...) \\ + (q^{1/8} - 7455q^{9/8} + 52416q^{17/8} - 33695q^{25/8} \pm ...) (e_{\ell} + e_{-\ell}) \\ + (464q^{1/2} + 9408q^{3/2} - 241920q^{5/2} + 69888q^{7/2} \pm ...) e_{2\ell}\end{gathered} \right); \\
H_{12} &= -\frac{1}{30} F_4 G_8 - \frac{1296}{5} \cdot \mathrm{Lift} \left( \begin{gathered} \; (96q + 23168 q^2 + 237888q^3 + 874752q^4 \pm ...) e_0 \\ + (q^{1/8} + 225q^{9/8} - 1344q^{17/8} - 2975q^{25/8} \pm ...) (e_{\ell} + e_{-\ell}) \\ + (-16q^{1/2} - 4032q^{3/2} - 80640q^{5/2} - 521472q^{7/2} \pm ...) e_{2\ell}\end{gathered} \right); \\
H_{14} &= -\frac{1}{14} F_6 G_8 + 6 F_4 H_{10} + \frac{2592}{7} \mathrm{Lift} \left( \begin{gathered} \; (96q - 48256q^2 - 1661376q^3 - 13838592q^4 \pm ...) e_0 \\ + (q^{1/8} - 519q^{9/8} - 8976q^{17/8} + 77785q^{25/8} \pm ...) (e_{\ell} + e_{-\ell}) \\ + (-16q^{1/2} + 7872q^{3/2} + 362880 q^{5/2} + 5161728q^{7/2} \pm ...) e_{2\ell} \end{gathered} \right); \\
H_{16} &= 16 F_6 H_{10} + \frac{537}{125} F_4 H_{12} + \frac{1}{6} G_8^2 + \frac{1}{75} F_4^2 G_8 - \frac{7776}{25} \Phi_{16}, \\ &\quad \Phi_{16} = \mathrm{Lift} \left( \begin{gathered} \; (96q + 46208q^2 + 6005568q^3 + 108655872q^4 \pm ...) e_0 \\ + (q^{1/8} + 465q^{9/8} + 54816q^{17/8} + 167185q^{25/8} \pm ...) (e_{\ell} + e_{-\ell}) \\ + (-16q^{1/2} - 7872q^{3/2} - 1082880 q^{5/2} - 28691712q^{7/2} \pm ...) e_{2\ell} \end{gathered} \right); \\
H_{18} &= -\frac{469}{25} F_4^2 H_{10} - 94 G_8 H_{10} + 13 F_6 H_{12} + 3 F_4 H_{14} + \frac{1}{3} G_8 G_{10} + \frac{1}{150}
F_4^2 G_{10} + \frac{7776}{25} \Phi_{18}, \\ &\quad \Phi_{18} = \mathrm{Lift} \left( \begin{gathered} \; (-48q - 128704q^2 - 16804896q^3 - 524507520q^4 \pm ...) e_0 \\ +(q^{1/8} + 105q^{9/8} + 65616q^{17/8} + 1609225q^{25/8} \pm ...) (e_{\ell} + e_{-\ell}) \\ + (8q^{1/2} + 1056q^{3/2} + 1287360q^{5/2} + 100856448q^{7/2} \pm ...) e_{2\ell} \end{gathered} \right).\end{align*}
\section{Ideal for component II}
\label{sec:ideal}
The ideal of the image curve~\eqref{eqn:curve_componentII} is minimally generated by 17 weighted-homogeneous relations, of weights
\[
12,12,14,14,16,16,16,16,16,18,18,18,18,18,20,20,20.
\]
The weight-12 and weight-14 equations are:
\begin{equation}
\label{eqn:weight12_14}
\begin{aligned}
F_4G_8 + 4F_4H_8 + 3G_{12} + 3H_{12} &= 0, \\
F_4^3 + 12F_4H_8 + \frac{27}{4}F_6^2 - \frac{27}{2}G_{12} + 9H_{12} &= 0, \\
F_4H_{10} - \frac{7}{4}F_6H_8 + \frac{1}{4}H_{14} &= 0, \\
F_4G_{10} - 3F_6G_8 + 3F_6H_8 + 6H_{14} &= 0.
\end{aligned}
\end{equation}
The weight-16 equations are:
\begin{equation}
\begin{aligned}
F_6H_{10} + \frac{240}{49}G_8H_8 + \frac{720}{49}H_8^2 + \frac{8}{49}H_{16} &= 0,\\
F_4H_{12} + \frac{412}{147}G_8H_8 + \frac{265}{49}H_8^2 + \frac{93}{98}H_{16} &= 0, \\
F_4^2H_8 - \frac{37}{49}G_8H_8 + \frac{330}{49}H_8^2 - \frac{9}{98}H_{16} &= 0, \\
F_4G_{12} - \frac{3}{4}F_6G_{10} - \frac{20}{49}G_8H_8 - \frac{60}{49}H_8^2 + \frac{81}{49}H_{16} &= 0, \\
G_8^2 - \frac{138}{49}G_8H_8 - \frac{855}{49}H_8^2 + \frac{162}{49}H_{16} &= 0.
\end{aligned}
\end{equation}
The weight-18 equations are:
\begin{equation}
\begin{aligned}
G_8H_{10} - 4G_{10}H_8 - 9H_8H_{10} &= 0, \\
F_4H_{14} - \frac{31}{10}G_{10}H_8 - \frac{69}{5}H_8H_{10} + \frac{123}{20}H_{18} &= 0, \\
F_6H_{12} + \frac{229}{30}G_{10}H_8 + \frac{107}{5}H_8H_{10} + \frac{31}{20}H_{18} &= 0, \\
F_4F_6H_8 - \frac{9}{10}G_{10}H_8 + \frac{9}{5}H_8H_{10} - \frac{3}{20}H_{18} &= 0,\\
G_8G_{10} - \frac{69}{5}G_{10}H_8 - \frac{252}{5}H_8H_{10} + \frac{81}{5}H_{18} &= 0.
\end{aligned}
\end{equation}
The weight-20 equations are:
\begin{equation}
\label{eqn:weight20}
\begin{aligned}
F_4H_8^2 + 3H_8H_{12} - \frac{45}{32}H_{10}^2 &= 0, \\
G_{10}H_{10} - 16G_{12}H_8 - 9H_{10}^2 &= 0, \\
F_4H_{16} - \frac{49}{216}G_{10}^2 + \frac{23}{9}G_{12}H_8 + \frac{41}{24}H_{10}^2 &= 0.
\end{aligned}
\end{equation}
\bibliographystyle{amsplain}
\bibliography{bib.general}{}
\end{document}